\documentclass[11pt]{amsart}

\usepackage{geometry}
\usepackage{amsfonts}
\usepackage{amssymb}
\usepackage[mathscr]{eucal}
\usepackage{amsmath}
\usepackage{amsthm}
\usepackage[all]{xy}
\usepackage[mathscr]{eucal}
\usepackage{color}
\definecolor{red-}{rgb}{1.0,0.0,0.0}
\definecolor{green-}{rgb}{0.0,0.7,0.0}
\definecolor{brown-}{rgb}{0.9,0.6,0.0}

\usepackage{hyperref}
\hypersetup{
    colorlinks=true,
    linkcolor=blue,
    citecolor=blue,
    filecolor=blue,
    urlcolor=blue
}

\newtheorem{defi}{Definition}[section]

\newtheorem{thm}[defi]{Theorem}
\newtheorem{cor}[defi]{Corollary}
\newtheorem{prop}[defi]{Proposition}
\newtheorem{lem}[defi]{Lemma}
\newtheorem{rem}[defi]{Remark}

\begin{document}

\title{Generalized polarized manifolds with low second class}

\author{Antonio Lanteri and Andrea Luigi Tironi}

%\date{October 16, 2018}
\date{\today}

\address{Dipartimento di Matematica ``F. Enriques'',
Universit\`a degli Studi di Milano, Via C. Saldini, 50,  I-20133 Milano,
Italy} \email{antonio.lanteri@unimi.it}
\address{
Departamento de Matem\'atica, Universidad de Concepci\'on, Casilla
160-C, Concepci\'on, Chile} \email{atironi@udec.cl}


\begin{abstract}
On a smooth complex projective variety $X$ of dimension $n$, consider an ample
vector bundle $\mathcal{E}$ of rank $r \leq n-2$ and an ample line bundle $H$. A numerical character
$m_2=m_2(X,\mathcal{E},H)$ of the triplet $(X,\mathcal{E},H)$ is defined, extending the well-known second class
of a polarized manifold $(X,H)$, when either $n=2$ or $H$ is very ample.
Under some additional assumptions on $\mathcal{F}: = \mathcal{E} \oplus H^{\oplus (n-r-2)}$,
triplets $(X,\mathcal{E},H)$ as above whose $m_2$ is small with respect to the invariants
$d:=c_{n-2}(\mathcal{F})H^2$ and $g:=1+\frac{1}{2}\big(K_X + c_1(\mathcal{F})+H\big) \cdot c_{n-2}(\mathcal{F}) \cdot H$ are
studied and classified.
\end{abstract}

\maketitle

\section*{Introduction}\label{Intro}

Let $S$ be a smooth complex projective surface embedded by a very
ample line bundle $L$. Identify $S$ with its image in
$\mathbb{P}^N$, $N=\dim H^0(S,L)-1$, via the embedding associated
with $L$ and think of the linear system $|L|$ corresponding to the
elements of $H^0(S,L)$ as the hyperplane linear system of $S$.
Consider also the dual variety $\mathcal{D}(S)$ of $S$, i.e. the
subset of $|L|$ parameterizing the tangent hyperplanes.
If $(S,L)\neq (\mathbb{P}^2,\mathcal{O}_{\mathbb{P}^2}(1))$, then $\mathcal{D}(S)$
is a hypersurface in the dual projective space $\mathbb{P}^{N \vee}$
(identified with $|L|$), and its degree $m$ is usually called the {\em class} of $S$.
More generally, for a projective manifold $X \subset \mathbb P^N$
one can consider its second class $m_2$, namely the class of its general surface section,
which is always positive, unless $X$ is a linear space, by what we said.
Like for the degree and the sectional genus, the study of $m_2$ contributed to a large
literature on the classification of smooth projective varieties with small invariants.
In particular, it is known that for $m \leq 29$, $S$ is a ruled
surface and pairs $(S,L)$
occurring for $m \leq 25$ are classified (see
\cite[p.195]{MG-K}, and \cite[Prop. 3.2]{TuV}).
Moreover, for $m\leq 11$ only
$(\mathbb{P}^2,\mathcal{O}_{\mathbb{P}^2}(e)), e=1,2$, and
scrolls may occur (e.g. see Remark \ref{veryample}).

Due to the fact that $m=c_2(J_1(L))$, the second Chern class of the first jet bundle of $L$,
in recent years the study of small values of $m$
for embedded surfaces has been reconsidered and transplanted in the wider setting of ample line bundles.
In particular, Palleschi and Turrini (\cite{PTu}) started to classify polarized surfaces $(S,H)$
when $H$ is only assumed to be ample on $S$ by studying small
values of $c_2(J_1(H))$ and of $c_2(J_1(H))-H^2$, in line with
classical papers by Marchionna \cite{M} and Gallarati
\cite{G1}, \cite{G2}. For pairs $(S,H)$ as above the situation
is different from the classical case because already for
$c_2(J_1(H))=5$ a non ruled surface occurs. Sometimes, in this context,
$m:=c_2(J_1(H))$ is referred to as the {\em generalized class} of the polarized surface $(S,H)$.

The aim of this paper is to revisit the study of this character in the framework of ample vector bundles.
We generalize $m_2$ from a projective manifold $X$ polarized by a very ample line
bundle $L$ to triplets $(X,\mathcal{E},H)$ in an appropriate vector bundle setting, and we study the
objects giving rise to small values of this character.
Roughly speaking, on a smooth complex projective variety $X$
of dimension $n$, consider an ample vector
bundle $\mathcal{E}$ of rank $r$, $2 \leq r \leq n-2$, and an ample line bundle
$H$. By considering the triplet
$(X,\mathcal{E},H)$
and the ample vector bundle of rank $n-2$ on $X$ given by
$\mathcal F: = \mathcal E \oplus H^{\oplus (n-r-2)}$, we
define the generalized class
$m_{2}=m_{2}(X,\mathcal{E},H)$ of $(X,\mathcal{E},H)$ as
\begin{equation}\label{*}
m_{2}:=\left[c_2(\Omega_X\oplus
\det \mathcal F)+c_1^2-c_2+H^2\right]\cdot
c_{n-2}+4(g-1), \tag{$*$}
\end{equation}
where $c_i:=c_i(\mathcal{F})$ for $i=1,2, \dots ,n-2$, and
$g:=1+\frac{1}{2}(K_X+c_1+H)\cdot H\cdot c_{n-2}$.
From now on we simply write
$m_{2}$ for $m_2(X,\mathcal E,H)$.

\smallskip
If $\mathcal{F}$ admits a section vanishing on a smooth surface $S$, it turns out that $m_2=c_2(J_1(H_S))$, the generalized class
of the polarized surface $(S,H_S)$. Moreover, for $H$ very ample
and $\mathcal{E}=H^{\oplus (n-2)}$, $m_2$ is just the second class of the projective manifold
$X$ embedded in $\mathbb{P}^{N}$ via $|H|$.

This allows us to revisit and extend several classification results for surfaces of
small class in the setting of ample vector bundles.
Actually, under the above assumption on $\mathcal{F}$,
we show that $m_{2}\geq d$, where
$d:=c_{n-2}(\mathcal F) \cdot H^2$, except for
$(X,\mathcal E,H)=(\mathbb P^n, \mathcal O_{\mathbb P^n}(1)^{\oplus r},
\mathcal O_{\mathbb P^n}(1))$, or $(\mathbb P^n, \mathcal O_{\mathbb P^n}(1)^{\oplus (n-2)},
\mathcal O_{\mathbb P^n}(2))$, and we describe completely the
triplets satisfying equality (see Theorem \ref{basic1}).

Then by putting $\delta:=m_{2}-d$,
in line with the classical case, we study small positive values of
$\delta$ by proving that $\delta\geq 6$, apart
from few triplets $(X,\mathcal E,H)$, which are precisely described (Theorem \ref{basic2}).
As a consequence of these results, we describe the possible
triplets $(X,\mathcal E,H)$ with $m_2\leq 6$. Moreover, we
carry on our analysis to prove that
if $m_{2} > 6$, then $m_{2} \geq 10$, provided that $S$ has
non-negative Kodaira dimension.
Including the sectional genus $g$ into the picture, we characterize triplets
for which $\delta \leq 2g+2$ (Proposition \ref{2q} and Theorem \ref{2g+1, 2g+2}) and we show that
$\delta \geq 2g+d$ if $S$ has non-negative Kodaira dimension.
Moreover, as expected, the stronger are the properties enjoyed by the line bundle $H_S$ (existence of a smooth
curve in $|H_S|$, spannedness by global sections, very ampleness), the larger
are the values of $m_2$ attained by our results. In particular, assuming that $H_S$ is
spanned by global sections, we list the triplets with $m_2 \leq 11$, those with $\delta \leq 2g+2$,
as well as those with $\delta \leq 2g+5$ provided that $S$ has non-negative Kodaira dimension
(Proposition \ref{basic5} and Theorem \ref{main}). In connection with this,
we have the opportunity to correct a mistake in \cite{LP2} (see Remark \ref{remark 4.6} ii)).
On the other hand, under the assumption that $H_S$ is very ample,
we revisit the above results and finally we prove that $\delta \geq 2g+11$ if $S$ has non-negative Kodaira dimension.

A great help in our analysis is provided by a number of results
on ample vector bundles having a section which vanishes on a surface
of some special kind (\cite{dF}, \cite{L2}, \cite{LM1}, \cite{LM2}, \cite{LM3}, \cite{LM4}). The strategy is the following: first, looking at the difference
$\delta$, which can be expressed in terms of geometric and topological characters,
we show, extending or refining some known results, that the polarized surface $(S,H_S)$ must belong to a
precise list of pairs. Next, by applying the results on ample vector bundles
mentioned before we succeed to reduce (sometimes drastically) these lists to a very short
number of cases, for which we obtain a rather complete description of $\mathcal{E}$ and $H$ according to
the admissible structure of $X$. For example, in some instances $S$ could ``a priori''
be a minimal elliptic surface, whose elliptic fibration turns out to be endowed with some
multiple fibers, but this possibility is ruled out by \cite{LM4}. Therefore these cases
do not lift to the vector bundle setting.

The paper is organized as follows. Section \ref{Sec0} contains
miscellaneous preliminary results on polarized surfaces $(S,L)$
with special regard to pairs for which $c_2(J_1(L))-L^2$ is small.
The starting point is the list in \cite[Theorem 4.3]{PTu},
combined with results of Fujita
\cite{F3} and Yokoyama \cite{Y}, which is summarized in Table \ref{Table1}.
In particular, in this setting, we prove new results holding either when $|L|$ contains a smooth
curve or when $S$ has non-negative Kodaira dimension, which will play a
relevant role in the sequel. In Section \ref{Sec1} the invariant
$m_2$ is introduced for triplets $(X,\mathcal{E},H)$ in an
appropriate setting and triplets for which $\delta$ is small are
analyzed. Moreover, lists of triplets with low $m_2$ are derived
from this study. In Section \ref{Sec1bis} significant bounds for
$\delta$ expressed in terms of the sectional genus $g$ are
discussed. Finally, in Section \ref{Sec2} all the above matter is
reconsidered under the extra assumption that the line bundle $H_S$
is ample and spanned (Subsection 4.1) or even very ample
(Subsection 4.2).

\smallskip

We work over the field of complex numbers and we use the standard notation and terminology from algebraic geometry.
In particular,

\smallskip

\begin{tabular}{llcl}
$\mathbb P^n$  & : &  &  the projective space of dimension $n$; \\
$\mathbb Q^n$  & : &  &  the smooth quadric hypersurface of $\mathbb
P^{n+1}$; \\
$\Omega_V$  & : &  &  the cotangent bundle of a smooth variety $V$;
\\
$q(V)$  & : &  &  the irregularity $h^1(\mathcal{O}_V)$ of
$V$;\\
$K_V$ & : & & the canonical bundle of $V$; \\
$\mathcal F_W$  & : &  &  the pull-back of a coherent sheaf
$\mathcal F$ on $V$ via an embedding $W\subset V$;
\end{tabular}

\begin{tabular}{llcl}
$(s)_0$  & : &  &  the (scheme-theoretic) zero locus of a section
$s$ of a vector bundle on $V$;\\
%\end{tabular}
%
%\begin{tabular}{rlcl}
$e(S)$  & : &  &  the topological Euler characteristic of a
surface $S$; \\
$\kappa(S)$  & : &  &  the Kodaira dimension of $S$;
\\
$g(S,\mathcal L)$  & : &  &  the sectional genus of a
polarized surface $(S,\mathcal L)$; \\
$\equiv$ & : & & the numerical equivalence relation.
\end{tabular}

\medskip

\noindent With a little abuse, we adopt the additive notation for the tensor product of line bundles.
We say that a smooth surface $S$ is ruled if it is birationally
ruled, i.e. if $\kappa(S) = - \infty$; $S$ is said to be
geometrically ruled if it is a $\mathbb{P}^1$-bundle over a smooth
curve. To denote a geometrically ruled surface of invariant $e$ over a
smooth curve of genus $q:=q(S)$ we use the non-standard symbol $S_{q,e}$ (in particular,
$S_{0,e}$ is the Segre--Hirzebruch surface of invariant $e$); however, as usual
(e.g., see \cite[Chapter V, \S\ 2]{H})
$C_0$ and $f$ will stand for a section of minimal
self-intersection $-e$ and a fiber, respectively. We recall that
$e\geq -q$ (Nagata inequality).

\vskip0.5cm

\section{Polarized surfaces $(S,L)$ with small class}\label{Sec0}

Here are some general facts concerning polarized surfaces $(S,L)$
which will be useful in the basic setting introduced in Section \ref{Sec1}.

\medskip

For the convenience of the reader, we sum up in Table \ref{Table1}
known results concerning polarized surfaces $(S,L)$
whose class $m:=c_2(J_1(L))$ is small compared to
the degree $d:=L^2$. We set $q:=q(S)$, $g=g(S,L)$, and we denote by $(S',L')$ a
minimalization of $(S,L)$, when $S$ is not minimal, as in \cite{PTu}. Recall that
letting $\eta:S \to S'$ be the corresponding birational morphism, we have $L = \eta^* L'- \sum \nu_iE_i$,
where $E_i$, $i=1, \dots, s$, are the exceptional curves contracted by $\eta$ and $\nu_i\geq 1$ for every $i$.

{%\small
{\tiny
\begin{table}[!h]
\begin{center}
  \begin{tabular}{ | c | c | c | c | c | c | c | c | c | }
    \hline
 N.              &  $m-d$          & $d$          & $g$         & $q$       & $\kappa(S)$       & $e(S)$     & $(S,L)$                                                                              & $(S',L')$ \\ \hline \hline
 1               &  -1             & 1, 4         & 0           & 0         & $-\infty$         & 3          & $(\mathbb{P}^2,\mathcal{O}_{\mathbb{P}^2}(e))$, $e=1,2$                              & -- \\ \hline
 2               &  0              & $\geq 1$     & $q$         & $\geq 0$  & $-\infty$         & $4-4q$     & scroll over a smooth curve of genus $q$                                              & -- \\ \hline
 3               &  3              & 9            & 1           & 0         & $-\infty$         & 3          & $(\mathbb{P}^2,\mathcal{O}_{\mathbb{P}^2}(3))$                                       & -- \\ \hline
 4               &  4              & 8            & 1           & 0         & $-\infty$         & 4          & $(\mathbb{P}^1\times\mathbb{P}^1,\mathcal{O}_{\mathbb{P}^1\times\mathbb{P}^1}(2,2))$ & -- \\ \hline
 5               & 4               & 8            & 1           & 0         & $-\infty$         & 4          & the blow-up at a point of                                                            & N. 3 \\ \hline
 6               & 4               & 3            & 2           & 1         & $-\infty$         & 0          & $(S_{1,-1},[3C_0-f])$                                                                & -- \\ \hline
 7               & 4               & 4            & 2           & 1         & $-\infty$         & 0          & $(S_{1,e},[2C_0+(e+1)f])$, $e=-1,0$                                                  & -- \\ \hline
 8               & 4               & 1            & 2           & 1         & 1                 & 0          & $S\to\mathbb{P}^1$ is a minimal elliptic surface                                     & -- \\
                 &                 &              &             &           &                   &            & with multiple fibers                                                                 &  \\ \hline
 9               & 4               & 2            & 2           & 2         & 0                 & 0          & $S$ is the Jacobian of a smooth curve $C$ of                                         & -- \\
                 &                 &              &             &           &                   &            & genus $2$, $L\equiv C$ embedded in $S$ and $h^0(L)=1$                                &  \\ \hline
 10              & 4               & 2            & 2           & 2         & 0                 & 0          & $S\cong C_1\times C_2$, $C_i$ is an elliptic curve                                   & -- \\
                 &                 &              &             &           &                   &            & for $i=1,2$, $L\equiv C_1+C_2$ and $h^0(L)=1$                                        &  \\ \hline
 11              & 4               & 2            & 2           & 2         & 0                 & 0          & $S$ is a bielliptic surface, $|L|=\{Z+F\},\ Z$ a section,                            & -- \\
                 &                 &              &             &           &                   &            & $F$ a fiber of the Albanese fibration                                                &  \\ \hline
 12              & 5               & 7            & 1           & 0         & $-\infty$         & 5          & the blow-up at two points of                                                         &  N. 3 \\ \hline
 13              & 5               & 1            & 2           & 2         & 0                 & 1          & the blow-up at a point of                                                            & N. 9 \\ \hline
 14              &  5              & 2            & 2           & 1         & $-\infty$         & 1          & the blow-up at a point of                                                            & N. 6 \\ \hline
 15              &  5              & 3            & 2           & 1         & $-\infty$         & 1          & the blow-up at a point of                                                            & N. 7 \\ \hline
 16              & 6               & 6            & 1           & 0         & $-\infty$         & 6          & the blow-up at three points of                                                       & N. 3 \\ \hline
 17              & 6               & 1            & 2           & 1         & $-\infty$         & 2          & the blow-up at two points of                                                         & N. 6 \\ \hline
 18              & 6               & 2            & 2           & 1         & $-\infty$         & 2          & the blow-up at two points of                                                         & N. 7 \\ \hline
 19              &  7              & 5            & 1           & 0         & $-\infty$         & 7          & the blow-up at four points of                                                        & N. 3 \\ \hline
 20              & 7               & 1            & 2           & 1         & $-\infty$         & 3          & the blow-up at three points of                                                       & N. 7 \\ \hline
 21              &  8              & 4            & 1           & 0         & $-\infty$         & 8          & the blow-up at five points of                                                        & N. 3 \\ \hline
 22              & 8               & 4            & 3           & 2         & 0                 & 0          & $S$ is an abelian surface                                                            & -- \\ \hline
 23              & 8               & 4            & 3           & 1         & 0                 & 0          & $S$ is a bielliptic surface                                                          & -- \\ \hline
 24              &  8              & $\leq$ 3     & 3           & $\geq 1$  & 1                 & 0          & $S$ is a minimal elliptic surface with $\chi (\mathcal{O}_S)=0$                      & -- \\ \hline
 25              & 8               & 12           & 2           & 0         & $-\infty$         & 4          & $(S_{0,e},[2C_0+(3+e)f]),$ with $e=0,1,2$                                            & -- \\ \hline
 26              & 8               & 8            & 3           & 1         & $-\infty$         & 0          & $(S_{1,e},[2C_0+(2+e)f]),$ with $e=-1,0,1$                                           & -- \\ \hline
 27              &  8              & 6            & 3           & 1         & $-\infty$         & 0          & $(S_{1,0},[3C_0+f])$                                                                 & -- \\ \hline
 28              &  8              & 5            & 3           & 1         & $-\infty$         & 0          & $(S_{1,-1},[5C_0-2f])$                                                               & -- \\ \hline
 29              & 8               & 4            & 4           & 2         & $-\infty$         & -4         & $(S_{2,e}, [2C_0+(e+1)f])$ with $-2\leq e\leq 0$                                     & -- \\ \hline
\end{tabular}

\medskip

\caption{Polarized surfaces $(S,L)$ with $m-d\leq 8$.}
\label{Table1}
\end{center}
\end{table}}

\smallskip

In particular, note that all pairs $(S,L)$ with $m\leq 9$ are included in Table \ref{Table1}.
The basic source for Table 1 is \cite[Section 4]{PTu}, taking into account
some progress in the classification of polarized surfaces of sectional genus two, compared to \cite{BLP}.
Moreover, as to N. 11, we note that
the description of $L$ provided in \cite[Theorem 2.7, d)]{BLP} has been improved by Fujita (see \cite[Theorem 15.7]{F4} and
\cite[Lemma 2.15]{F2}); $|L|$ consists of a single divisor, which is the sum of a section and a fiber of the Albanese fibration.
As a consequence, \cite[Remark 2.3 (2)]{Y} implies that no simple blow-up of a pair as in N. 11 can occur.
Furthermore, the results concerning ruled surfaces
over an elliptic curve, due to Fujita \cite[\S\ 4]{F3} (see also \cite[Theorem 15.2, cases 0), and 3)--5)]{F4}) and Yokoyama
\cite[Theorem 4.1 (ii)]{Y}, lead to a
simplification in \cite[Theorem 4.3]{PTu}. For instance, combining both we see that for $g=2$ and $S'=S_{1,e}$,
it must be $\nu_i=1$ for every $i$, hence $e(S) = s = {L'}^2 - d$.

\smallskip

\begin{rem}\label{veryample}
{\em If $L$ is very ample, the only surviving cases in Table \ref{Table1}
are N. 1--5, 12, 16, 19, 21, 25, and 26
with $e=-1$.}
\end{rem}

\medskip

It is useful to recall that for a polarized surface $(S,L)$ we
have $m = e(S)+2K_SL + 3 L^2$, hence
\begin{equation}\label{formula}
m-d = e(S) + 2K_SL + 2d = e(S) + 4(g-1) \ . \tag{$\#$}
\end{equation}

\begin{lem} \label{lemma}
Let $S$ be a smooth surface, $L$ an ample line bundle on $S$, and
let $g:=g(S,L)$ be the sectional genus of $(S,L)$. Suppose that
$\sigma:S \to S_0$ is the blow-up of a smooth surface $S_0$ at a
single point and let $E$ be the exceptional curve. Then there
exists an ample line bundle $L_0$ on $S_0$ such that $L=\sigma^*
L_0 - rE$, where $r = LE \geq 1$. Moreover, $L^2 = L_0^2 - r^2$,
$L K_S = L_0 K_{S_0} + r$. In particular,
\begin{enumerate}
\item[i)] $g = g(S_0,L_0) - \binom{r}{2}$ $($hence $g=g(S_0,L_0)$ if
and only if $r=1)$; \item[ii)] If $L K_S = 1$, then $S$
cannot have Kodaira dimension $\kappa(S) \geq 1$.
\end{enumerate}
\end{lem}
\begin{proof}
The Nakai--Moishezon criterion proves the ampleness of $L_0$.
Assertion i) is obvious since $K_S = \sigma^* K_{S_0} + E$. To
prove assertion ii) note that
$$1 = L K_S = L_0 K_{S_0} + r.$$
We know that $r \geq 1$. If $\kappa(S) \geq 1$, then a suitably
high multiple of $K_{S_0}$ is effective and nontrivial and then
also the first summand on the right hand is positive, due to the
ampleness of $L_0$, but this gives a contradiction.
\end{proof}

Note that ampleness and spannedness of $L$ imply $h^0(L) \geq 3$
and $L^2 \geq 3$ up to well known cases.
More precisely, we have also the following

\begin{lem} \label{doublecovers}
Let $L$ be an ample and spanned line bundle on a smooth surface
$S$. Then $d=L^2 \geq 3$ unless $(S,L, e(S),g,m,m-d)$ is one of
the following:
\begin{enumerate}
\item[i)] $(\mathbb P^2, \mathcal O_{\mathbb P^2}(1), 3, 0,
0,-1)$; \item[ii)] $(\mathbb Q^2, \mathcal O_{\mathbb Q^2}(1), 4,
0, 2,0)$; \item[iii)] There exists a morphism $\pi:S \to \mathbb
P^2$ of degree $2$, branched along a smooth curve $\Delta \in
|\mathcal O_{\mathbb P^2}(2b)|$ for some integer $b \geq 2$
$($case $b=1$ fits into case {\em ii}$))$; moreover,
$L=\pi^*\mathcal{O}_{\mathbb P^2}(1)$, $e(S) = 2(2b^2-3b+3)$, $g=b-1$,
$m=2b(2b-1)\geq 12$ and $m-d=2(2b^2-b-1)\geq 10$.
\end{enumerate}
\end{lem}

\begin{proof} It is enough to consider the morphism defined by $|L|$ and
recall that $L^2$ is the product of its degree and the degree of
the image. In case iii) note that $\pi^*|\mathcal O_{\mathbb
P^2}(1)| = |L|$ (since $b \geq 2$). Recall that $\pi_*\mathcal O_S
= \mathcal O_{\mathbb P^2} \oplus \mathcal O_{\mathbb P^2}(-b)$.
Since $\Delta\in |\mathcal O_{\mathbb P^2}(2b)|$ and $K_S =
\pi^*(K_{\mathbb P^2} + \frac{1}{2}\Delta) = \pi^*(\mathcal
O_{\mathbb P^2}(b-3))$, projection formula gives
$$h^0(K_S) = h^0(\pi_*K_S) = h^0(\mathcal O_{\mathbb P^2}(b-3)\oplus
\mathcal O_{\mathbb P^2}(-3))= h^0(\mathcal O_{\mathbb P^2}(b-3))= \binom{b-1}{2}.$$
Similarly, $h^1(K_S)=0$ and then, since $K_S^2 = 2(b-3)^2$,
Noether's formula allows us to compute $e(S)$. The value of $g$ is
provided by the Riemann--Hurwitz formula, by restricting $\pi$ to
a general element of $\pi^*|\mathcal O_{\mathbb P^2}(1)|$.
\end{proof}

The following fact will be used often.

\begin{rem}\label{well-known}
{\em Let $(S,L)$ be a smooth polarized surface of sectional genus $g\geq 2$. If $S$ is ruled, but $(S,L)$ is not a scroll, then
$g \geq 2q$. Actually, due to the assumptions, $K_S+L$ is nef, hence $(K_S+L)^2 \geq 0$. Moreover, $K_S^2 \leq 8(1-q)$. Combining these
inequalities we get
$$0 < L^2 \leq L^2 + (K_X+L)^2 = 2(K_X+L)L + K_S^2 \leq 4(g-1) + 8(1-q) = 4(g-2q+1).$$
Therefore $g > 2q-1$.}
\end{rem}

Now, observe that for a polarized surface $(S,L)$, the inequality
$m-d\geq 2g$ in \cite[Proposition 3.2]{PTu} can be further explored
by assuming that there exists a smooth curve in the linear system $|L|$, as the following result shows.

\begin{thm}\label{prop 2g+1, 2g+2}
Let $(S,L)$ be a smooth polarized surface and put $m:=c_2(J_1(L))$.
Assume that $|L|$ contains a smooth curve. Then
\begin{enumerate}
\item[(A)] $m-d=2g+1$ if and only if either
\begin{itemize}
\item[$(\alpha)$] $(m-d, g)=(3,1),(5,2)$ and $(S,L)$ is as in {\em
Table \ref{Table1}}, or \item[$(\beta)$] $g=2q\geq 4$, $S$ is the
blowing-up $\sigma :S\to S_{q,e}$ of $S_{q,e}$ at a point $p$,
$L=\sigma^*L'-\sigma^{-1}(p)$ and $L'\equiv [2C_0+(e+1)f]$.
\end{itemize}
\item[(B)] $m-d=2g+2$ if and only if either
\begin{itemize}
\item[$(\gamma)$] $(m-d, g)=(4,1), (8,3)$ and $(S,L)$ is as in {\em Table \ref{Table1}}, or
\item[$(\delta)$] $g\geq 4$ and $(S,L)$ is one of the following polarized surfaces:
\begin{enumerate}
\item[$(\delta_1)$] $S=S_{q,e}$ with $q\geq 2$, $g=2q+1$, $L\equiv [2C_0+(e+2)f]$ and $d=8$;
\item[$(\delta_2)$] $S=S_{2,-1}$, $g=5$, $L\equiv [3C_0-f]$ and $d=3$;
\item[$(\delta_3)$] $S$ is the
blowing-up $\sigma :S\to S_{q,e}$ of $S_{q,e}$ at two points $p_1, p_2$, lying on distinct fibers, $g=2q$,
$L=\sigma^*L'-\sigma^{-1}(p_1)-\sigma^{-1}(p_2)$, $L'\equiv [2C_0+(e+1)f]$ and $d=2$.
\end{enumerate}
\end{itemize}
\end{enumerate}
\end{thm}

\begin{proof}
(A) We can assume $g \geq 4$, since otherwise $m-d \leq 7$, hence $(S,L)$ is as in Table \ref{Table1}.
Then \eqref{formula} implies $0=e(S)+2g-5\geq e(S)+3$, hence
$S$ is a ruled surface. Note that
$(S,L)$ is not a scroll since $m-d\neq 0$, hence $g\geq 2q$ by Remark \ref{well-known}. Let $\sigma:S\to S'$ be the blowing-up of a smooth ruled surface $S'$ at a finite set of points $\mathfrak{B} \subset S'$. Denote by $s$ the cardinality of $\mathfrak{B}$.
Thus $e(S)=4(1-q)+s$ and this gives $0=2(g-2q)+(s-1)$. Observe that necessarily $s \leq 1$, and $s=0$
cannot occur. Hence $s=1$, i.e. $S$ is the blowing-up of $S'$ at a single point, and $g=2q$. As a consequence, $q \geq 2$. Consider the smooth curve $C\in |L|$
and the morphism $p:C \to B$ obtained by restricting to $C$ the ruling projection of $S$ onto its base curve $B$.
By the Riemann--Hurwitz formula we obtain that
$$4q-2=2g-2=r(2q-2)+b,$$ where $r$ and $b$ are the degree and the total branching order
of $p$, respectively.
Therefore, since $(S,L)$ is not a scroll, we get
$$2\leq r=\frac{4q-2-b}{2q-2}=2+\frac{2-b}{2q-2},$$ which implies that either (i) $b=r=2$, or (ii)
$b=0, q=2$ and $r=3$. Moreover, write $L=\sigma^*L'- \nu E$, where $E$ is the exceptional divisor contracted by $\sigma$, $L'$ is an ample divisor on $S'$ and $\nu$ is a positive integer.
Let $L'\equiv [rC_0+\beta f]$ and note that
$$1 \leq (\sigma^*f-E)(\sigma^*L'- \nu E)=(\sigma^*f-E)(\sigma^*(rC_0+\beta f)- \nu E) = r - \nu,$$ i.e.
$1\leq \nu \leq r-1$. Thus in case (i) we have $\nu =1$ and if $S'$ has invariant $e$ then by the
genus formula we deduce that
$$4q-2=2g-2=(K_S+L)L=(K_{S'}+L')L'=4(q-1)+(2\beta -2e),$$ i.e. $\beta =e+1$. This gives case $(\beta)$ in the statement.

Finally, in case (ii) we see from $q=2$, $r=3$, $b=0$ that $g=4$ and $1\leq \nu \leq 2$. By the
genus formula, we get the following relation
$$\nu (\nu - 1) = 2(2\beta -3e).$$
According to it, $\nu=1$ would imply $2\beta - 3e=0$, but this contradicts the ampleness of $L'$. On the other hand,
$\nu=2$ would imply $2\beta-3e=1$, hence $4 = \nu^2 < L^2 + \nu^2 = {L'}^2  = 3(2\beta - 3e)=3$, a contradiction. Thus case (ii) cannot occur.

\bigskip

\noindent (B) As in case (A), we can suppose that $g\geq 4$. Then $0=e(S)+2g-6\geq e(S)+2$, hence $S$ is a ruled surface.
Moreover, note that $(S,L)$ is not a scroll over a curve. Hence $g\geq 2q$ by Remark \ref{well-known} again and using the same notation as in (A),
we can write $e(S)=4(1-q)+s$ for some integer $s\geq 0$. This gives $0=2(g-2q)+(s-2)$, i.e. $0\leq s\leq 2$. Thus we have either (j) $s=0, g=2q+1$, or (jj) $s=2, g=2q$.

In case (j), $S=S_{q,e}$ is a geometrically ruled surface over a smooth curve $B$. Note that $4 \leq g = 2q+1$ implies
$q \geq 2$. We can write $L \equiv [aC_0+bf]$ with $a \geq 2$, since $(S,L)$ is not a scroll. Then
\begin{equation} \label{L^2}
L^2 = a(2b-ae).
\end{equation}
If $a=2$ then $(S,L)$ is a conic bundle and
$$d=L^2=2(K_S+L)L+K_S^2=4(g-1)+8(1-q)=4(g-2q)+4=8,$$
which gives $L \equiv [2C_0+(e+2)f]$; moreover, $-q \leq e \leq 1$ in view of the
Nagata inequality and the ampleness conditions \cite[p.\ 382]{H}. This gives case $(\delta_1)$ of the statement.
Next, assume that $a\geq 3$. Recalling the expression of $K_S$,
we obtain
$$(K_S+L)^2 = \big((a-2)C_0+(b-e+2q-2)f\big)^2 = (a-2)\big(2b-ae + 4(q-1)\big) \geq 4q-3,$$
in view of the ampleness of $L$. So
\begin{eqnarray} \label{long}
L^2 \leq L^2+(K_S+L)^2-4q+3=2(K_S+L)L+K_S^2-4q+3\\
=4(g-1)+8(1-q)-4q+3=8q+8-8q-4q+3 = 11-4q \leq 3, \notag
\end{eqnarray}
as $q \geq 2$. Combining this with \eqref{L^2} and the ampleness conditions, we get
$$3 \geq L^3 = a(2b-ae) \geq 3 (2b-ae) \geq 3,$$
and therefore $a=3$, $b=\frac{1}{2}(3e+1)$, $d=3$, and all inequalities in \eqref{long}
are equalities. Thus $q=2$; moreover, since $e$ has to be odd, Nagata inequality implies
$e=-1$, hence $b=-1$. This gives case $(\delta_2)$ in the statement.

\medskip

In case (jj), $S$ is obtained by a blowing-up $\sigma :S\to S'$ of a geometrically ruled surface $S'\to \Gamma$ over a smooth curve $\Gamma$ at two points $p_1$ and $p_2$.
Denote by $E_i$ the
corresponding exceptional divisor for $i=1,2$. Thus $L=\sigma^*L'- \nu_1E_1- \nu_2E_2$ for an ample line bundle $L'=[aC_0+bf]$ on $S'$ with $a\geq 2$ and positive integers $\nu_i$ for $i=1,2$.
If $a=2$, the ampleness of $L$ implies that $p_1$ and $p_2$ lie on distinct fibers and $\nu_i=1, i=1,2$. From
$$4q-2=2g-2=(K_S+L)L=(K_{S'}+L')L'=4(q-1)+2(b-e)$$
we get $b=e+1$, hence $e\leq 0$ in view of the ampleness conditions. This gives case $(\delta_3)$ in the statement.
Now let $a \geq 3$. By applying the Riemann--Hurwitz formula to the $a:1$ cover $C\to \Gamma$, we see that $3\leq a\leq\frac{2g-2}{2q-2}=\frac{2q-1}{q-1}=2+\frac{1}{q-1}$. Hence $a=3$, $q=2$ and from $1\leq (\sigma^*f-E_i)L=3-\nu_i$,
it follows that $\nu_i\leq 2$, that is, $\nu_i=1,2$. Since $q=2$, we have $g=4$ and by the genus formula we obtain that
$$6=2g-2=(K_S+L)L=(a-2)(b-ea)+a(b+2-e)-\nu_1(\nu_1-1)-\nu_2(\nu_2-1)=$$
$$=b-3e+3(b+2-e)-\nu_1(\nu_1-1)-\nu_2(\nu_2-1),$$
i.e. $$(2b-3e)=\frac{1}{2}\big( \nu_1(\nu_1-1)+\nu_2(\nu_2-1) \big).$$
Note that $0<d=3(2b-3e)-\nu_1^2-\nu_2^2$, but this leads to a numerical contradiction.
\end{proof}

\medskip

If $S$ is not a ruled surface, a result of Serrano \cite{Se} allows us to go further.

\begin{prop}\label{S not ruled}
Let $(S,L)$ be a smooth polarized surface and put
$m:=c_2(J_1(L))$. Suppose that $m-d>0$ and assume that $S$ is not
a ruled surface. Then $m-d \geq 2g+d$ unless one of the following cases occurs:
\begin{enumerate}
\item $S$ is an abelian or a bielliptic surface and $m-d = 2g+d-2$;
\item $S$ is an elliptic quasi-bundle $f:S\to B$
over a smooth curve $B$ of genus $g(B)\leq 1$, $q=1$, $p_g(S)=0$
and $m-d = 2g+d-1$; moreover,
$f$ has only multiple fibers $m_iF_i$, $i=1,..,s$, as singular fibers, where $F_i$ is a smooth elliptic curve, and letting $F$ denote
the general fiber of $f$,
one of the following holds:
\begin{enumerate}
\item $g(B)=1$, $s=1$, $m_1=2$ and $FL=2$ $($e.g., see  {\em \cite{F1}}$)$;
\item $g(B)=0$ and $(m_1,...,m_s)=(2,2,2,2,2), (4,4,4), (2,6,6)$ with $FL=2,4,6$ respectively $($e.g., see  {\em \cite{Se}}$)$.
\end{enumerate}
\end{enumerate}
\end{prop}

\begin{proof}
Assume that $m-2d\leq 2g-1$. Then
$$e(S)+2g-2+K_SL=m-2d\leq 2g-1,$$ i.e.
$e(S)+K_SL\leq 1$. Note that
$e(S)\geq 0$ and
$K_SL\geq 0$ since $S$ is not a ruled surface. Thus we get the following three cases:

(i) $e(S)=K_SL=0$ and $m-2d = 2g-2$;

(ii) $e(S)=0, K_SL=1$ and $m-2d = 2g-1$;

(iii) $e(S)=1, K_SL=0$ and $m-2d = 2g-1$.

\noindent Case (iii) cannot occur: actually, it follows from
$K_SL=0$ that $K_S$ is numerically trivial, since $S$ is not a ruled surface.
Therefore $S$ is a minimal surface with $\kappa(S)=0$, but this contradicts
$e(S)=1$.
\noindent
In cases (i) and (ii), $S$ is a minimal surface with $\kappa(S)\leq 1$, since $e(S)=0$.
If $\kappa(S)=1$ then a multiple of the canonical bundle is nontrivial and effective, but
this contradicts (i) by the ampleness of $L$. Moreover, in case (ii), since $K_S$ is not numerically trivial,
we see that $S$ is a properly elliptic minimal
surface over a smooth curve $B$, hence $K_S^2=0$. Thus $\chi (\mathcal{O}_S)=0$, by the Noether's formula.
Then, by \cite[Proposition 4.2]{Se1}, the elliptic fibration $f:S\to B$ is a quasi-bundle, i.e. any singular fiber is a
multiple of a smooth elliptic curve \cite[Definition 1.1]{Se1}. By combining the canonical bundle formula
for an elliptic fibration with the condition $LK_S=1$, it thus follows that $f$ necessarily has some multiple
fiber and $g(B)\leq 1$. Moreover, $0 < p_g(S)+1 = q = g(B)$ or $g(B)+1$ \cite[\S 4]{Se1}, but the latter case
cannot occur if $g(B)=1$, due to the Katsura--Ueno property \cite[Proposition 4.3]{Se1}.
Then the assertion follows from
\cite{Se}, taking into account that this result only depends on the condition $LK_S=1$ (and
not $g=2$), as pointed out in \cite[final comment at p.\ 300]{Se1}.
\end{proof}

\begin{rem}\label{general type}
{\em Let $S$ be a surface of general type. Then, by combining
Noether\rq s formula with the Bogomolov--Miyaoka--Yau inequality,
we have $e(S) \geq 3$.}
\end{rem}

\vskip0.5cm

\section{Triplets $(X,\mathcal E,H)$ with low $m_2$}\label{Sec1}

\medskip

Our basic setting from here on is the following:
\bigskip
\begin{equation}\label{0}\tag{$\Diamond$}
X\ \text{\rm{is a smooth complex projective variety of dimension}}\ n,\ \mathcal{E}\ \text{\rm{is an ample vector bundle}}
\end{equation}
\begin{equation}\notag
\text{\rm{of rank $r$ on $X$ with $2 \leq r\leq n-2$ and $H$ is an
ample line bundle on $X$. Furthermore, the}}
\end{equation}
\begin{equation}\notag
\text{\rm{ample vector bundle of rank $n-2$ on $X$ given by
$\mathcal F: = \mathcal E \oplus H^{\oplus (n-r-2)}$ has a section
vanishing }}
\end{equation}
\begin{equation}\notag
\noindent \text{\rm{on a smooth surface $S \subset
X$.}}\hspace{11.3cm}
\end{equation}

\bigskip

\begin{rem}\label{rem diamond}
{\em A concrete way to fit into
\eqref{0} for $r < n-2$ is to consider the following slightly more
special setting: $X$ is a smooth complex projective variety of dimension $n, \mathcal{E}$ is an ample vector bundle
of rank $r$ on $X$ with $2 \leq r < n-2$, having a section whose zero locus is a smooth subvariety $Z \subset X$ of the expected dimension $n-r$ (which happens, e.g., if $\mathcal{E}$ is spanned), and $H$ is an ample line bundle on $X$
such that $\text{\rm{Tr}}_Z|H|$ (the trace of $|H|$ on $Z$) is base point free.}

\noindent {\em Note that in this setting the line bundle
$H_Z$ is spanned ``a fortiori".
Clearly this fits into \eqref{0} simply letting $S$ denote
the surface cut out by $n-r-2$ general elements of
$\text{\rm{Tr}}_Z|H|$. Actually, if $\sigma \in \Gamma(X,\mathcal
E)$ defines $Z$, there are sections $s_i \in \Gamma(X,H)$ whose
restrictions to $Z$ define a smooth surface
$$ S:= \bigcap_{i=1}^{n-r-2}({{s_i}|_Z})_0\ ,$$
which is the zero locus of the section $(\sigma, s_1, \dots, s_{n-r-2}) \in \Gamma(X, \mathcal F)$. }
\end{rem}

\medskip

In Subsection $4.1$ we will add to \eqref{0} the requirement that
\begin{equation} \label{S}
H_S \ \text{\rm{is spanned}}. \tag{S}
\end{equation}
Clearly, this condition is trivially satisfied in the
setting of Remark \ref{rem diamond} since, as noted,
$H_Z$ is spanned. Furthermore, in Subsection $4.2$ we will put the
stronger requirement that
\begin{equation} \label{VA}
H_S \ \text{\rm{ is very ample}}. \tag{VA}
\end{equation}

\medskip

Assuming that $(X,\mathcal{E}, H)$ is
as in \eqref{0}, we set
$$d:=H_S^2=H^2\cdot c_{n-2}=c_r(\mathcal E) \cdot H^{n-r} \qquad{\text{\rm{and}}}\qquad g:=g(S,H_S).$$
This notation is consistent with that used in Section \ref{Sec0} since
$d=H_S^2$ and $g$ are
the degree and the sectional genus of the polarized surface $(S,L):=(S,H_S)$, respectively. Moreover,
we have the following technical result.

\begin{prop}\label{proposition}
Let $(X,\mathcal{E}, H)$ and $S$ be as in \eqref{0} (see Introduction). If
$m_2=m_2(X,\mathcal{E}, H)$ is as in \eqref{*}, then
$$m_{2}=
c_2(J_1(H_S)).$$
\end{prop}

\begin{proof}
Consider the dual of the tangent--normal bundle sequence of $S \subset X$
$$0\to \mathcal{N}_{S/X}^\vee\cong {\mathcal{F}^\vee}_S\to
(\Omega_X)_S\to\Omega_S\to 0. $$
It fits into the following diagram
$$\begin{matrix}
& & 0 & & & & & & \\
& & \downarrow & & & & & & \\
& & \mathcal{N}_{S/X}^\vee\otimes H_S & \cong & {\mathcal{F}^\vee}_S\otimes H_S & & & &\\
& & \downarrow & & & & & &\\
0 & \to & ({\Omega_X}\otimes H)_S & \to & J_1(H)_S & \to & H_S & \to & 0 \\
& & \downarrow & & & & & &\\
0 & \to & {\Omega_S}\otimes H_S & \to & J_1(H_S) & \to & H_S & \to & 0 \\
& & \downarrow & & & & & &\\
& & 0 & . & & & & &\\
\end{matrix}$$
Then, recalling $(*)$ we get
$$m_{2}=m_{2}(X,\mathcal E,H)=\left[c_2(\Omega_X\oplus c_1)+c_1^2-c_2+H^2\right]_S+4(g-1)=$$
$$=\left[c_2({\Omega_X}_S)+c_1({\Omega_X}_S)c_1(\mathcal{F}_S)+c_1(\mathcal{F}_S)^2-c_2(\mathcal{F}_S)\right]
+H_S^2+$$
$$+4(g(S,H_S)-1)=\left[c_2({\Omega_X}_S)-c_1(\Omega_S)c_1(\mathcal{F}^\vee_S)-c_2(\mathcal{F}^\vee_S)\right]+H_S^2+$$
$$+4(g(S,H_S)-1)=c_2(\Omega_S)+H_S^2+2(K_S+H_S) H_S=$$
$$=\left[c_2(\Omega_S)+c_1(\Omega_S)H_S+H_S^2\right]+\left[(K_S+2H_S) H_S\right]=$$
$$=c_2(\Omega_S\otimes H_S)+c_1(\Omega_S\otimes H_S)
H_S=c_2(J_1(H_S)).$$
\end{proof}

\begin{lem} \label{lemma 1}
Suppose that there exists an effective divisor
$E\cong\mathbb{P}^{n-1}$ on $X$ such that
$$(E,\mathcal{F}_E,E_E)\cong (\mathbb{P}^{n-1},\mathcal{O}_{\mathbb{P}^{n-1}}(1)^{\oplus (n-2)},
\mathcal{O}_{\mathbb{P}^{n-1}}(-1)).$$ If $q>0$, then $S$
is not ruled.
\end{lem}

\begin{proof} Suppose that $\kappa(S)=-\infty$ and note that $S\neq\mathbb{P}^2$. Let
$f:X\to X'$ be the contraction of $E$. Then by \cite[Lemma
5.1]{LM2} and \cite[Lemma 2.2]{L2} we know that there exist an
ample vector bundle $\mathcal{F}'$ of rank $n-2$ on $X'$ and a
section $s'\in\Gamma(\mathcal{F}')$ such that $\mathcal{F}\cong
f^*\mathcal{F}'\otimes\mathcal{O}_X(-E)$, $S':=(s')_0$ is a smooth
surface and $f|_S:S\to S'$ is a birational morphism which
contracts the $(-1)$-curve $E|_S$. Since
$K_{S'}=[K_{X'}+\det\mathcal{F}']_{S'}$ is not nef, $S'$ being
ruled, and $q(S')>0$,
\cite[Theorem]{Ma} implies one of the following possibilities:
\begin{enumerate}
\item[(i)] there exists an effective divisor $E'$ on $X'$ such that
$$(E',\mathcal{F'}_{E'},E'_{E'})\cong (\mathbb{P}^{n-1},
\mathcal{O}_{\mathbb{P}^{n-1}}(1)^{\oplus
(n-2)},\mathcal{O}_{\mathbb{P}^{n-1}}(-1));$$
\item[(ii)] there is a surjective morphism $\varphi:X' \to W$ expressing $X'$ either (a) as
a $\mathbb P^t$-bundle
over a smooth variety $W$ of dimension $\dim W \leq 2$, or
(b) as a quadric fibration over a smooth curve $W$.
\end{enumerate}
We show that case (ii) does not occur.
If (ii) holds, then
$X'$ is covered by lines. Note
that any line of $X'$ is contained in a fiber of $\varphi$ since
$q(X')=q(S')=q>0$. Suppose that $p':=f(E)$ lies on a smooth
fiber of $X'$, take a line $l'$ passing through $p'$ and consider
its proper transform $l$ via $f$. Then
$$\mathcal{F}_l\cong (f^*\mathcal{F}'\otimes\mathcal{O}_X(-E))_l\cong
f^*(\mathcal{F}'_{l'})\otimes\mathcal{O}_l(-1)\cong
f^*(\mathcal{F}_{\mathbb{P}^t}')_{l'}\otimes\mathcal{O}_l(-1)$$
$$\cong f^*(\mathcal{O}_{l'}(a)\oplus\mathcal{O}_{l'}(1)^{\oplus (n-3)})\otimes\mathcal{O}_l(-1)\cong
\mathcal{O}_l(a-1)\oplus\mathcal{O}_l^{\oplus (n-3)},$$ for
$a=1,2$ \cite[Theorem, cases (10), (11), (13)]{Ma}. But this contradicts the ampleness of $\mathcal{F}$ since
$n\geq 4$. On the other hand, if we are in case (ii)(b) and $p'$
is contained in a singular fiber, then we have $\deg
\mathcal{F'}_{l'}= \deg \mathcal{F'}_{\lambda'}=n-2$ for some line
$\lambda'$ contained in a smooth fiber \cite[Theorem, case (12)]{Ma}. By the ampleness of
$\mathcal{F}'$, we get
$\mathcal{F'}_{l'}=\mathcal{O}_{l'}(1)^{\oplus (n-2)}$. Thus the
same argument as above with $a=1$ applies and this shows that this
case cannot occur as well. Finally, if we are in case (i), by a
recursive argument we get a contradiction.
\end{proof}

\begin{lem} \label{lemma 2}
Let $(X,\mathcal E, H)$ be as in \eqref{0}, let $\mathcal F=
\mathcal E \oplus H^{\oplus (n-r-2)}$, and suppose that $S$ is a
$\mathbb P^1$-bundle over a smooth curve $B$ of positive genus.
Then $X$ is a $\mathbb P^{n-1}$-bundle over $B$, with the
projection $p:X \to B$ inducing the ruling of $S$, and $\mathcal
F_F = \mathcal O_{\mathbb P^{n-1}}(1)^{\oplus (n-2)}$ for every
fiber $F \cong \mathbb P^{n-1}$ of $p$. In particular, either $r <
n-2$ and $(S,H_S)$ is a scroll, or $r=n-2$ and $H_F = \mathcal
O_{\mathbb P^{n-1}}(t)$ with $t = H_S f$, $f$ being any fiber of
$S$. Conversely, if $(X, \mathcal{F})$  satisfies the above
conditions, then $S$ is a $\mathbb{P}^1$-bundle
 over $B$; moreover, $(S,H_S)$ is a scroll if either $r < n-2$ or $H_F = \mathcal O_{\mathbb P^{n-1}}(1)$.
\end{lem}
\begin{proof}
The description of $(X,\mathcal F)$, including the fibration $p:X \to B$, follows from
\cite[Theorem]{LM2}. If $r < n-2$, then $H_F = \mathcal O_{\mathbb
P^{n-1}}(1)$, being a summand of $\mathcal F_F$, and then
$(S,H_S)$ is a scroll. On the other hand, if $r=n-2$ then
$\mathcal F = \mathcal E$, so we have no information on $H$. We
can write $H_F = \mathcal O_{\mathbb P^{n-1}}(t)$ for some
positive integer $t$. Since the ruling of $S$ is induced by $p:X
\to B$ any fiber $f$ of $S$ is a line, being the zero locus of a
section of $\mathcal E_F=\mathcal O_{\mathbb P^{n-1}}(1)^{\oplus
(n-2)}$, where $F=\mathbb P^{n-1}$ is the corresponding fiber of
$X$. Thus the assertion follows from the equality
$$H_S f = H_F \cdot \big(\mathcal O_{\mathbb P^{n-1}}(1)\big)^{n-2} = t.$$
The converse is obvious.
\end{proof}

\medskip

Recall the notation $\delta := m_2-d$. As a first thing, let us
characterize the inequality $\delta < 0$.

\begin{thm} \label{basic1}
Let $(X, \mathcal{E}, H)$ be as in \eqref{0}.
Then $$\delta \geq 0$$
unless $(X,\mathcal E,H)$ is either $(\mathbb P^n, \mathcal
O_{\mathbb P^n}(1)^{\oplus r}, \mathcal O_{\mathbb P^n}(1))$
$(m_{2}=0)$, or $(\mathbb P^n, \mathcal O_{\mathbb P^n}(1)^{\oplus
(n-2)}, \mathcal O_{\mathbb P^n}(2))$ $(m_{2}=3)$. Moreover,
equality holds if and only if $(X,\mathcal E,H)$ is one of the
following:
\begin{enumerate}
\item $(\mathbb P^n, \mathcal O_{\mathbb P^n}(2) \oplus \mathcal
O_{\mathbb P^n}(1)^{\oplus (r-1)}, \mathcal O_{\mathbb P^n}(1))$;
\hfill $(m_{2}=2)$ \item $(\mathbb Q^n, \mathcal O_{\mathbb
Q^n}(1)^{\oplus r}, \mathcal O_{\mathbb Q^n}(1))$;  \hfill
$(m_{2}=2)$ \item $X$ is a $\mathbb P^{n-1}$-bundle over a smooth
curve $B$, $\mathcal E_F = \mathcal O_{\mathbb P^{n-1}}(1)^{\oplus
r}$ and $H_F = \mathcal O_{\mathbb P^{n-1}}(1)$, for every fiber
$F = \mathbb P^{n-1}$ of the bundle projection $\pi : X\to B$ and
$(S,H_S)$ is a scroll over $B$ via $\pi|_S :S\to B$.
\hfill $(m_{2}=d:=H_S^2)$
\end{enumerate}
\end{thm}

\begin{proof}
Since \eqref{0} holds, by Proposition \ref{proposition} and
\cite[Proposition (A.1)]{LPS} we see that
$\delta=c_2(J_1(H|_S))-d\geq 0$ except for (a)
$(S,H_S)=(\mathbb{P}^2, \mathcal O(e))$ with $e=1,2$, and
$\delta=0$ holds if and only if (b) $(S,H_S)$ is a scroll over a
smooth curve.

In (a), by \cite[Theorem A]{LM1} we know that $X=\mathbb{P}^n$ and
$\mathcal F=\mathcal O(1)^{\oplus (n-2)}$, which
gives rise to the first two triplets in the statement.

In (b), by \cite[Theorem 2]{LM3} we see that $(X,\mathcal F)$ is one
of the following pairs:
\begin{enumerate}
\item[(i)] $(\mathbb{P}^n,\mathcal O(1)^{\oplus (n-3)}\oplus\mathcal
O(2))$;
\item[(ii)] $(\mathbb{Q}^n,\mathcal O(1)^{\oplus (n-2)})$;
\item[(iii)] $X$ is a $\mathbb P^{n-1}$-bundle over a smooth curve
$B$ and $\mathcal F_F = \mathcal O_{\mathbb P^{n-1}}(1)^{\oplus
(n-2)}$ for every fiber $F = \mathbb P^{n-1}$ of the bundle
projection.
\end{enumerate}
Since in this situation $S\neq\mathbb{P}^2$, cases (i) and (ii)
give (1) and (2) of the statement with $\mathcal{E}=\mathcal
O(1)^{\oplus (r-1)}\oplus\mathcal O(2)$ and $\mathcal{E}=\mathcal
O(1)^{\oplus r}$ respectively and $H=\mathcal O(1)$ in both cases.
Finally, (iii) leads easily to case (3) of the statement.
\end{proof}

\medskip

The following result characterizes the low positive values of
$\delta$.

\begin{thm} \label{basic2}
Let $(X,\mathcal{E},H)$ be as in \eqref{0}, and suppose
that $\delta$ is
positive. Then
$$\delta\geq 3$$
with equality if and only if
$(X, \mathcal E, H) = (\mathbb P^n, \mathcal O_{\mathbb P^n}(1)^{\oplus (n-2)}, \mathcal O_{\mathbb P^n}(3))\ (m_{2}=12).$

Moreover, if $\delta = 4$ then $(X, \mathcal E, H)$ is one of the
following triplets:
\begin{enumerate}
\item $(\mathbb P^n, \mathcal O_{\mathbb P^n}(2) \oplus \mathcal O_{\mathbb P^n}(1)^{\oplus (n-3)}, \mathcal O_{\mathbb P^n}(2))$;
\hfill $(m_{2}=12)$
\item $(\mathbb Q^n, \mathcal O_{\mathbb Q^n}(1)^{\oplus (n-2)}, \mathcal O_{\mathbb Q^n}(2))$; \hfill
$(m_{2}=12)$
\item $r=n-2$, $X$ is a $\mathbb P^{n-1}$-bundle over
a smooth curve $B$ of genus $1$, $\mathcal E_F = \mathcal O_{\mathbb P^{n-1}}(1)^{\oplus (n-2)}$ and
$H_F = \mathcal O_{\mathbb P^{n-1}}(t)$,
with $t=2$ or $3$, for every fiber $F = \mathbb P^{n-1}$ of the bundle projection
$X\to B$; moreover, $(S,H_S)$ is,
up to numerical equivalence, either $(S_{1,-1},[3C_0-f])\ (m_{2}=7)$
or $(S_{1,e},[2C_0+(e+1)f])$ with $e\in\{-1,0\}$
 $(m_{2}=8)$.
\end{enumerate}

Finally, if $\delta = 5$ then $(X, \mathcal E, H)$ is one of the
following triplets:
\begin{enumerate}
\item[(4)] there is a vector bundle $\mathcal{T}$ on a smooth
curve $C$ of genus one such that
$X\cong\mathbb{P}_C(\mathcal{T})$, $H_F=\mathcal{O}_{F}(1)$ and
$\mathcal{E}_F\cong\mathcal{O}_{F}(2)\oplus\mathcal{O}_F(1)^{\oplus
(r-1)}$
for any fiber $F\cong\mathbb{P}^{n-1}$ of $X\to C$;
moreover, $S$ is the blowing-up $\sigma:S\to S_{1,e}$ of
$S_{1,e}\to C$ with $e\in\{-1,0\}$ at a point $p$ and
$H_S\equiv [\sigma^*(2C_0+(e+1)f)-\sigma^{-1}(p)]$;  \hfill $(m_{2}= 8)$
\item[(5)] there is a surjective morphism $q:X\to \Gamma$ onto a
smooth curve $\Gamma$ of genus one such that any general fiber $F$
of $q$ is a smooth quadric hypersurface of $\mathbb{P}^n$ with
$H_F=\mathcal{O}_F(1)$ and
$\mathcal{E}_F\cong\mathcal{O}_F(1)^{\oplus r}$;  moreover, $S$ is
the blowing-up $\sigma:S\to S_{1,e}$ of $S_{1,e}\to C$ with
$e\in\{-1,0\}$ at a point $p$ and
$H_S\equiv [\sigma^*(2C_0+(e+1)f)-\sigma^{-1}(p)]$; \hfill $(m_{2}=8)$
\item[(6)] $X=\mathbb{P}_{\Sigma}(\mathcal{U})$, where $\Sigma$ is
the Jacobian of a smooth curve $\gamma$ of genus $2$,
$\mathcal{U}$ is an ample vector bundle of rank $n-1$ over
$\Sigma$ and $\mathcal{E}=\pi^*\mathcal{G}\otimes
\xi$, where $\xi$ is the tautological line bundle on $X$,
$\mathcal{G}$ is a vector bundle of rank $r$ on $\Sigma$ and
$\pi:X\to \Sigma$ is the bundle projection; moreover,
$H_F=\mathcal{O}_F(t)$ for any fiber $F\cong\mathbb{P}^{n-2}$ of
$\pi$ with $t\geq 1$ and $t=1$ if $r<n-2$, $\pi|_S:S\to \Sigma$ is
the blowing-up of $\Sigma$ at a point $p$ and
$H_S=\pi|_S^*\gamma-\pi|_S^{-1}(p)$, looking at the curve $\gamma$ as embedded in its jacobian.
\hfill $(m_{2}=6)$
\end{enumerate}

\end{thm}

\begin{proof} It follows from Table \ref{Table1} that
$\delta = m_2 - d \geq 3$, with equality if and only if $(S,H_S) = (\mathbb
P^2, \mathcal O_{\mathbb P^2}(3))$. By \cite[Theorem 4 and Remark
in Sec. 2]{LM3}, this pair leads to the first assertion in the
statement.

So we continue supposing that $m_2 - d \geq 4$. Now, assume that
equality holds. Taking into account the pairs $(S,H_S)$ in Table
\ref{Table1}, we see that condition $g \leq 1$ forces
$(S,H_S)$ to be either $(\mathbb P^1 \times \mathbb P^1, \mathcal
O(2,2))$ or $(S_{0,1}, -K_{S_{0,1}})$. In both cases
$(S,H_S)$ is a del Pezzo pair, but \cite[Theorem 4 and Remark at
the end of $\S 2$]{LM3} shows that only the former case lifts to
the vector bundle setting giving rise to (1) and (2) in the
statement. Next assume $g \geq 2$. If $S$ is not ruled,
according to $(\#)$ and the interpretation of $m_2$ (see
Proposition \ref{proposition}), the equality $m_2 - d=4$ implies
$g=2$ and $e(S)=0$. In this case, by \cite[Theorem
4.2]{PTu}, $S$ is a minimal surface, which is either i) an
elliptic fibration over $\mathbb P^1$ with some multiple fibers
(see \cite{Se}), or ii) an abelian or a bielliptic surface. By
\cite[Theorem]{LM4} case i) cannot occur: actually, the fact that
$S$ is minimal contradicts \cite[Theorem (a)]{LM4}
while the existence of multiple fibers is in contrast with
\cite[Theorem (b)]{LM4}.
Similarly, case ii) cannot occur since the only minimal
surface of Kodaira dimension zero occurring as zero locus of an
ample vector bundle is a K$3$ surface \cite[Theorem]{L2}.
Therefore $S$ is a ruled surface, and then, according to
\cite[Theorem 4.3]{PTu} $S$ is a $\mathbb P^1$-bundle over an
elliptic curve; moreover, $g=2$ and one of the following
cases holds:
\begin{enumerate}
\item[(a)] $S:=S_{1,-1}$ and $H_S \equiv [3C_0-f]$; \item[(b)]
$S:=S_{1,e}$ with $e\in\{-1,0\}$ and $H_S \equiv [2C_0+(e+1)f]$.
\end{enumerate}
Since $S$ is an irrational $\mathbb P^1$-bundle, we can use \cite[Theorem]{LM2}
to conclude that $X$ is a $\mathbb P^{n-1}$-bundle over a smooth curve $B$ and
$\mathcal F_F = \mathcal O_{\mathbb P^{n-1}}(1)^{\oplus (n-2)}$ for every
fiber $F$ of the bundle projection $\pi:X \to B$. This implies that
$\pi|_S$ is the bundle projection of $S$, $f$ being a line in $F$, hence $B$ is the elliptic base curve of $S$.
Moreover, we see that if $r < n-2$, then $H_F = \mathcal O_{\mathbb P^{n-1}}(1)$, as a summand of $\mathcal F_F$,
but this is in contradiction with the fact that
$1 = H f = H_S f = 2$ or $3$,
according to cases (a) and (b). Thus $r = n-2$ and $H_F = \mathcal O_{\mathbb P^{n-1}}(t)$, with $t=2$ or $3$. This gives (3) in
the statement.

Finally, assume that $m_2 - d =5$. If $g \leq 1$, then equality holds and
$S$ is the blowing-up of $\mathbb{P}^2$ at two points with $H_S=-K_S$, $q=p_g=0$ and $e(S)=5$.
Since $(S,H_S)$ is a del Pezzo surface, this situation cannot lift to the ample vector bundle setting by
\cite[Theorem 4 and Remark at the end of $\S 2$]{LM3}.
Thus $g \geq 2$. From
Table \ref{Table1} we know that
$g=2$ and $(S,H_S)$ is one of the following pairs:
\begin{enumerate}
\item[(i)] $\kappa(S)\geq 0$, $H_S^2=1$ and $S$ is the blowing-up
at a single point of the Jacobian of a smooth curve $C$ of genus $2$;
\item[(ii)] $S$ is
ruled, $q=1$ and $S$ is the blowing-up $\sigma :S\to S'$ at a
point $p$ of a $\mathbb{P}^1$-bundle $S'$ over a smooth curve $B$
of genus $1$ and one of the following conditions holds:

(c) $S'=S_{1,-1}$ and $H_S\equiv [\sigma^*(3C_0-f)-\sigma^{-1}(p)]$;

(d) $S'=S_{1,e}$ with $e=0,-1$ and
$H_S\equiv [\sigma^*(2C_0+(e+1)f)- \sigma^{-1}(p)]$.

\end{enumerate}

\noindent In case (i), since $S$ is birationally equivalent to an abelian
surface, by \cite[Theorem]{L2} we obtain case (6) in the statement. In case
(ii), since $S$ is a non-minimal ruled surface, it follows from
\cite[Theorem]{Ma} and Lemmas \ref{lemma 1} and \ref{lemma 2} that
$(X,\mathcal{F},H)$ is one of the following triplets:
\begin{enumerate}
\item[(j)] there is a vector bundle $\mathcal{V}$ on a smooth curve $C$ such that $X\cong\mathbb{P}_C(\mathcal{V})$ and $\mathcal{F}_F\cong\mathcal{O}_F(2)\oplus\mathcal{O}_F(1)^{\oplus (n-3)}$ for any fiber $F\cong\mathbb{P}^{n-1}$ of $X\to C$;
\item[(jj)] there is a surjective morphism $X\to C$ onto a smooth curve $C$ such that any general fiber $F$ of $X\to C$ is a smooth quadric hypersurface of $\mathbb{P}^n$ with
$\mathcal{F}_F\cong\mathcal{O}_F(1)^{\oplus (n-2)}$;
\item[(jjj)] there is a vector bundle $\mathcal{U}$ on a smooth surface $\Sigma$ such that $X\cong\mathbb{P}_{\Sigma}(\mathcal{U})$ and $\mathcal{F}_F\cong\mathcal{O}_F(1)^{\oplus (n-2)}$ for any fiber $F\cong\mathbb{P}^{n-2}$ of $\pi:X\to\Sigma$.
\end{enumerate}
Write $H_F=\mathcal{O}_F(t)$ for some positive integer $t$.
In cases (j) and (jj), note that the fibration $X\to C$ restricted to $S$ is the ruling projection $S\to B$, hence $C \cong B$. Moreover, we have $F_S=\sigma^*f$ for a general fiber $F$, since $F_S\cdot\sigma^*f=0$ for any general fiber $f$ of $S'$, $g(F_S)=0$ and $F_S^2=0$.
Since
$$H_S\cdot\sigma^*f=H_S\cdot F_S=H_F\cdot S_F=H_F\cdot c_{n-2}(\mathcal{F})_F=H_F\cdot c_{n-2}(\mathcal{F}_F)=2t$$ is even for any fiber $f$ of $S'$ and
\begin{equation} \notag
H_S\cdot\sigma^*f= \begin{cases}
2 \quad \mathrm{in\ case}\ $(d)$ \\
3 \quad \mathrm{in\ case}\ $(c)$,
\end{cases}
\end{equation}
we conclude that only case (d) can occur with $t=1$.
This leads to cases (4) and (5) in the statement.

In (jjj), note that case (d) cannot occur by \cite[Theorem]{dF}. Moreover, in case (c) we have
\begin{equation} \notag
[K_X+\det (\mathcal{F}\oplus H)]_S=K_S+H_S \equiv [\sigma^*(C_0)].
\end{equation}
Therefore,
$K_X+\det (\mathcal{F}\oplus H)$ is not ample.
So, by \cite[Theorem C)]{ABW} we know that there exist a morphism
$s:X\to W$ expressing $X$ as a smooth projective $n$-fold $W$ blown-up at a finite
set $\Gamma\neq\emptyset$ and an ample vector bundle $\mathcal{F}'$ on $W$ such that
$\mathcal{F}\oplus H=s^*\mathcal{F}'\otimes [-s^{-1}(\Gamma)]$ and $K_W+\det
\mathcal{F}'$ is ample. Consider an exceptional divisor $E\cong\mathbb{P}^{n-1}$ of $s$.
Since $n-1\geq 3$, we see that $\pi(E)$ is a point of $S'$, but this is impossible since
any fiber of $\pi$ is a linear $\mathbb{P}^{n-2}$. Therefore, case (c) cannot occur.
\end{proof}

\medskip

\begin{rem}
{\em From Theorem \ref{basic2}
we deduce that $1$ and $2$ are gap values for $\delta$. Thus apart from a short list of
triplets $(X,\mathcal E,H)$ as in Theorem \ref{basic2}, we have
$\delta\geq 6$.}
\end{rem}

\medskip

\begin{rem} \label{Rem-example}
{\em
Let us note here that case $(6)$ in Theorem \ref{basic2} is effective.
Recall that this case comes
from case $(13)$ of \cite[Theorem]{Ma}. Let $(C, o)$ be a pointed
smooth curve of genus $2$, and on the jacobian $J(C)$ of $C$
consider the Jacobian bundle $\mathcal E_r(C,o)$ of rank $r$, as
in \cite[(2.18)]{F2}. Set $X:=\mathbb
P(\mathcal E_{n-1}(C,o))$; then $X$ is a $\mathbb P^{n-2}$-bundle
over the smooth surface $J(C)$. Recall that $X$ can be identified
with $C^{(n)}$, the $n$-fold symmetric product of $C$, the bundle
projection $\pi:C^{(n)} \to J(C)$ being given by the mapping
$(x_1, \dots ,x_n) \mapsto [x_1+ \dots +x_n - no]$. Let $H$ be the
tautological line bundle on $X$; $H$ is ample. Moreover, as shown
in \cite[(2.18)]{F2} there is a section of $H$ whose zero locus is
$\mathbb P(\mathcal E_{n-2}(C,o))$, which can be identified with
$C^{(n-1)}$. By induction, we thus see that $S:=C^{(2)}$ is the
zero locus of a section of the ample vector bundle $H^{\oplus
(n-2)}$. Thus the triplet $(X, \mathcal E:= H^{\oplus (n-2)}, H)$
provides an example as in case $(6)$ of Theorem \ref{basic2}. Note
also that $\pi|_S:S \to J(C)$ is just the contraction of the
unique $(-1)$-line of $(S, H_S)$ corresponding to the canonical
$g^1_2$ of $C$. }
\end{rem}

\bigskip

\noindent To avoid long lists repeating several triplets we already met, in
the next statement, as well as in Section \ref{Sec2},  we simply
denote by

\bigskip
\medskip

\noindent $\mathcal{A}:$ \ the class consisting of the five triplets
appearing in Theorem \ref{basic1};

\bigskip
\medskip

\noindent $\mathcal{B}:$ \ the class
consisting of the first three triplets occurring in Theorem
\ref{basic2}, namely,

$(\mathbb P^n, \mathcal O_{\mathbb
P^n}(1)^{\oplus (n-2)}, \mathcal O_{\mathbb P^n}(3))$, $(\mathbb
P^n, \mathcal O_{\mathbb P^n}(2) \oplus \mathcal O_{\mathbb
P^n}(1)^{\oplus (n-3)}, \mathcal O_{\mathbb P^n}(2))$, and
$(\mathbb Q^n, \mathcal O_{\mathbb Q^n}(1)^{\oplus (n-2)},
\mathcal O_{\mathbb Q^n}(2))$.

\bigskip
\medskip

\noindent Thus by Theorems \ref{basic1} and \ref{basic2}, combined with Table
\ref{Table1} and the fact that $m_2=\delta+d\geq\delta +1$, we have the following consequence.

\bigskip

\begin{cor} \label{cor basic1}
Let $(X, \mathcal{E}, H)$ be as in \eqref{0}.
Then $$m_2 \leq 6$$ if and only if either $(X,\mathcal
E,H)\in\mathcal{A}$\ $(m_{2}=0,3,2,2,d\leq 6)$, or $(X,\mathcal
E,H)$ is as in case $(6)$ of
{\em Theorem
\ref{basic2}}\ $(m_{2}=6)$.
\end{cor}

As a consequence of Corollary \ref{cor basic1}, we have $m_2 \geq
7$ apart from a short list of triplets $(X, \mathcal E, H)$.

\medskip

\noindent Finally, in line with Corollary \ref{cor basic1}, we
show that also 7--9 are gap values for $m_2$, provided that $S$ is
not ruled.

\begin{thm} \label{thm basic3}
Let $(X, \mathcal{E}, H)$ be as in \eqref{0} and suppose
that $S$ is not a ruled surface. If $m_2 \geq 7$, then
$$m_2 \geq 10.$$ Moreover,
if equality holds, then $r=n-2$,
$X=\mathbb{P}_{\Sigma}(\mathcal{U})$, for an ample vector bundle
$\mathcal{U}$ of rank $n-1$ on a smooth minimal surface $\Sigma$,
and $\mathcal{E}=\pi^*\mathcal{G}\otimes \xi$, where $\xi$ is the
tautological line bundle on $X$, $\mathcal{G}$ is a vector bundle
of rank $n-2$ on $\Sigma$ and $\pi :X\to \Sigma$ is the bundle
projection; furthermore, $H_F=\mathcal{O}_F(3)$ for any fiber
$F\cong\mathbb{P}^{n-2}$ of $\pi$, $\pi|_S:S\to \Sigma$ is the
blowing-up of $\Sigma$ at a point $p$, $e(S)=1$ and $\Sigma$ is
either an abelian or a bielliptic surface with
$H_S=\pi|_S^*L_0-3\pi|_S^{-1}(p)$ and $L_0^2=10$.
\end{thm}

\begin{proof}
Since $m_2 \geq 7$ and $S$ is not ruled, we see from Table \ref{Table1}
(where $m$ is now our $m_2$) that $m_2\geq 9$ and equality
implies that $S$ is a minimal elliptic surface, $\chi
(\mathcal{O}_S)=0$, $g=3$ and $H_S^2=1$. In this case, it
follows from \cite[Theorem(b)]{LM4} that $X$ is endowed with a
morphism $\varphi :X\to B$ onto a smooth curve $B$ inducing on $S$
the elliptic fibration and $f:=\varphi|_S:S\to B$ has no multiple
fibres. Thus by \cite[(12.1) and (12.2) in Chapter V, pp. 161--162]{BPV} we
deduce that $K_S=f^*(K_B)\equiv (2g(B)-2)F$ since $\deg
f_{*1}(\mathcal{O}_S)^{\vee}=\chi (\mathcal{O}_S)=0$, where $F$ is
a fiber of $f$, but this gives the numerical contradiction
$$2(g(B)-1)FH_S=K_SH_S=2g-2-H_S^2=4-1=3.$$

Suppose now that $m_2=10$. Then by $(\#)$ we have
$10=e(S)+4(g-1)+H_S^2\geq 4(g-1)+1$, i.e $g=2,3$ and $e(S)\leq 5$.
Assume that $\kappa(S)=2$. Then by \cite[Proposition 2]{AP} we
deduce that $g\leq 2$, hence
$g=2$. From \cite[Theorem 1.4]{BLP} we have $K_S\equiv L$ and $K_S^2=1,
q=0, p_g=0,1,2$. Thus we get
$$10=e(S)+4(g-1)+H_S^2=12\chi
(\mathcal{O}_S)-K_S^2+4+H_S^2=12(1+p_g)+3+H_S^2\geq 16,$$ but this
is a contradiction. Hence $\kappa(S)=0$ or $1$.

Suppose that $\kappa(S)=1$. If $g=2$, then by \cite[Lemma
1.1]{BLP} we know that $K_SH_S=H_S^2=1$ and $S$ is minimal.
Moreover, we have either (i) $\chi (\mathcal{O}_S)=0$ or (ii)
$q=p_g(S)=0$. By the Hodge inequality, we see that
$K_S^2=K_S^2H_S^2\leq (K_SH_S)^2=1$. Thus in (ii) we get
$12=12(1-q+p_g)=e(S)+K_S^2\leq 5+1=6$, a contradiction. In case
(i), by \cite[Lemma 1.3, Theorem 1.5]{BLP} we have $q=1,
p_g(S)=0$ and $S\to\mathbb{P}^1$ is an elliptic fibration with
multiple fibers, but this is in contradiction with \cite{LM4}.

If $g=3$ then $2=e(S)+H_S^2$, i.e. either
\begin{enumerate}
\item[(j)] $(e(S),H_S^2)=(0,2)$, or
\item[(jj)] $(e(S),H_S^2)=(1,1)$.
\end{enumerate}

In case (j), we see that $S$ is minimal
since $e(S)=0$. Let $f:S\to B$ be the elliptic fibration over a
smooth curve $B$. Then from \cite[Lemma VI.4]{B1} we deduce that
$f$ has no singular fiber. Furthermore, by \cite[(12.1) and
(12.2) in Chapter V, pp.\ 161--162]{BPV} we obtain that $K_S=f^*(K_B)\equiv
(2g(B)-2)F$ since $\deg f_{*1}(\mathcal{O}_S)^{\vee}=\chi
(\mathcal{O}_S)=0$. Hence $2=K_SH_S=(2g(B)-2)FH_S$, i.e. $g(B)=2$
and $FH_S=1$. This implies that $h^0(H_S)=0$, since
$g=3$. Thus the Riemann--Roch theorem shows that
$h^1(H_S)=h^0(H_S)=0$ and by the
exact sequence
$$0\to H_S \to H_S+F \to (H_S)_F \to 0,$$ we obtain that
$h^0(H_S+F)=h^0((H_S)_F)=1$. Let $C\in
|H_S+F|$. Note that $CF=(H_S+F)F=H_SF=1$ and $g(S,C)=4$. This
shows that $C$ is a reduced divisor. Moreover, observe that $C$ is also
irreducible, since otherwise $C-F=H_S$ would be effective. Thus from
$CF=1$ and $g(B)=2$, it follows that $g(S,C)=g(C)=2$, a contradiction.

In case (jj), note that $\sigma :S\to S_0$ is the blowing-up of a
minimal elliptic surface $S_0$ at a point $p$ with $\chi
(\mathcal{O}_{S_0})= \chi (\mathcal{O}_{S})=0$. This implies that
$K_{S_0}\equiv (2g(B)-2)F$, where $f:S_0\to B$ is the elliptic
fibration over a smooth curve $B$ and $F$ is a fiber of $f$. Note
that $S_0$ has no multiple fiber, otherwise $S$ itself would
contain multiple fibers, which is impossible in view of \cite[Theorem]{LM4}. Write $H_S=\sigma^*H_0-aE$, where $a$ is
a positive integer and $E$ is the exceptional curve of $\sigma$.
Thus
$$(2g(B)-2)FH_0+a=K_{S_0}H_0+a=K_SH_S=2g-2-H_S^2=4-H_S^2,$$
i.e.
\begin{equation} \label{new_display}
(2g(B)-2)FH_0+a=3.
\end{equation}
Observe that $1\leq
(\sigma^*F'-E)H_S=H_0F'-a=H_0F-a$, i.e. $a\leq H_0F-1$, where $F'$
is the fiber of $f$ passing through $p$. This shows that $H_0F\geq
2$ and then $g(B)=1$ by \eqref{new_display}. Hence $K_{S_0}\equiv\mathcal{O}_{S_0}$, but
this is impossible since
$\kappa(S)=1$.

Finally, assume that $\kappa(S)=0$. If $g=2$, then $6=e(S)+H_S^2$
and by \cite[Theorem 2.7]{BLP} we know that $K_S^2=0,-1$. Thus
$12\chi (\mathcal{O}_S)=e(S)+K_S^2\leq 5$, i.e. $\chi
(\mathcal{O}_S)=0$ and $e(S)=0,1$. So we get $H_S^2=6-e(S)=5,6$,
but this contradicts \cite[Proposition 2.1]{BLP}. Hence $g=3$ and
$2=e(S)+H_S^2$. This implies that $(e(S), H_S^2)$ is either (I)
$(0,2)$ or (II) $(1,1)$. In case (I) we see that $S$ is minimal with $\chi
(\mathcal{O}_S)=0$,
but this contradicts \cite[Theorem]{L2}. In (II), we deduce that
$S$ is the blowing-up $\sigma :S\to S_0$ of an abelian or a
bielliptic surface $S_0$ at a point $p$. Write
$H_S=\sigma^*H_0-a\sigma^{-1}(p)$ for some positive integer $a$.
Thus
$3=K_SH_S=(\sigma^*K_{S_0}+\sigma^{-1}(p))(\sigma^*H_0-a\sigma^{-1}(p))=a$,
i.e. $H_S=\sigma^*H_0-3\sigma^{-1}(p)$. By \cite[Theorem]{L2} we know that
$\pi : X\to S_0$ is a $\mathbb{P}^{n-2}$-bundle over $S_0$ and $\mathcal F_F
=\mathcal O_F(1)^{\oplus (n-2)}$ for any fiber $F\cong\mathbb{P}^{n-2}$ of
$\pi$. But in our case $H_F=\mathcal{O}_F(3)$ and then $r=n-2$
because $\mathcal F = \mathcal E \oplus H^{\oplus (n-r-2)}$.
\end{proof}

As a consequence of Corollary \ref{cor basic1} and Theorem
\ref{thm basic3}, when $S$ is not a ruled surface, we conclude
that $m_2 \geq 11$ apart from a short list of triplets.

\vskip0.5cm

\section{Lower bounds for $\delta$ in terms of $g$}\label{Sec1bis}

In this section, we will compare $\delta=m_2-d$ with the sectional
genus $g$ of the polarized surface $(S,H_S)$. A first result is
given by the following

\begin{prop}\label{2q}
Let $(X, \mathcal{E}, H)$ be as in \eqref{0}.
Then $$\delta\geq 2g$$ and equality holds if and only if $r=n-2$,
$X$ is a $\mathbb{P}^{n-1}$-bundle over a smooth curve $B$ and
$\mathcal{E}_F\cong\mathcal{O}_F(1)^{\oplus (n-2)}$ for every fiber
$F\cong\mathbb{P}^{n-1}$ of the bundle projection $\pi :X\to B$.
In particular, $S=S_{q,e}$ is a $\mathbb P^1$-bundle over $B$ via
$\pi|_S$. Moreover, either
\begin{enumerate}
\item[(i)] $H_F=\mathcal{O}_F(3)$, $q=1$, $e=-1$, $g=2$ and
$H_S\equiv [3C_0-f]$, or \item[(ii)] $H_F=\mathcal{O}_F(2)$,
$g=2q>0$ and $H_S\equiv [2C_0+(e+1)f]$ with $-q\leq e\leq
0$.
\end{enumerate}
\end{prop}

\begin{proof} By \cite[Proposition 3.2]{PTu}, we know that $\delta\geq 2g$, equality holding
if and only if one of the following cases occurs:
\begin{enumerate}
\item $S=S_{1,-1}$, $g=2$ and $H_S\equiv [3C_0-f]$; \item
$S=S_{q,e}$ with $-q\leq e\leq 0$, $g=2q>0$ and $H_S\equiv
[2C_0+(e+1)f]$; \item $S$ is a minimal surface endowed with an
elliptic fibration $S\to\mathbb{P}^1$, $q=1, p_g=0, g=2$ and
$H_S^2=1$; \item $S$ is a minimal and not ruled surface with
$g=2$.
\end{enumerate}
Note that in case (3), from \cite[Theorem 1.5]{BLP} it follows
that $S$ has multiple fibers, but this contradicts
\cite[Theorem]{LM4}. Moreover, also case (4) cannot occur by
\cite[Theorem]{L2} since $S$ is minimal and not a K3 surface.
Finally, by \cite[Theorem]{LM2} in cases (1) and (2) we conclude
that $X$ is a $\mathbb{P}^{n-1}$-bundle over a smooth curve $B$
and $\mathcal{F}_F\cong\mathcal{O}_F(1)^{\oplus (n-2)}$ for every
fiber $F\cong\mathbb{P}^{n-1}$ of the bundle projection $\pi :X\to
B$. Note that $F_S=f$ for any fiber $F$ of $\pi$ and that the
restriction $\pi|_S :S\to B$ of $\pi$ to $S$ gives the bundle
projection on $S$. Moreover, we have $H_F=\mathcal{O}_F(b)$ with
$b=3,2$ according to cases (1) and (2) respectively. This shows
that necessarily $r=n-2$.
\end{proof}

\noindent Now, we lift the results of Theorem \ref{prop 2g+1, 2g+2} and Proposition \ref{S not ruled} to the ample vector bundle setting.

\medskip

As a consequence of Theorem \ref{prop 2g+1, 2g+2}, we can obtain
the following

\begin{thm}\label{2g+1, 2g+2}
Let $(X, \mathcal{E}, H)$ be as in \eqref{0}. Moreover,
assume that there exists a smooth curve in $|H_S|$.
\begin{itemize}
\item[(A)] If $\delta = 2g+1$, then either $(\delta,g)=(3,1),(5,2)$ and the triplets $(X,\mathcal{E},H)$ fit into
all the possibilities of {\em Theorem
\ref{basic2}} for $\delta=3$ and $5$, or $(X, \mathcal E, H)$ is one of
the following triplets:
\begin{enumerate}
\item[(i)] $X\cong\mathbb{P}_C(\mathcal{V})$, where $\mathcal{V}$
is a vector bundle of rank $n$ on a smooth curve $C$,
$\mathcal{E}_F\cong\mathcal{O}_{F}(2)\oplus\mathcal{O}_F(1)^{\oplus
(r-1)}$ and $H_F=\mathcal{O}_{F}(1)$, for any fiber
$F\cong\mathbb{P}^{n-1}$ of the bundle projection $X \to C$;
\item[(ii)] there is a
surjective morphism $X\to \Gamma$ onto a smooth curve $\Gamma$
whose general fiber $F$ is a smooth
quadric hypersurface of $\mathbb{P}^n$
such that $\mathcal{E}_F\cong\mathcal{O}_F(1)^{\oplus r}$
and $H_F=\mathcal{O}_F(1)$.
\end{enumerate}
Moreover, in both cases, $S$ is the blowing-up $\sigma :S\to
S_{q,e}$
of a surface $S_{q,e}$ at a point $p$, $H_S\equiv [\sigma^*(2C_0+(e+1)f)-\sigma^{-1}(p)]$ and $g=2q\geq 4$.\\
\item[(B)] If $\delta = 2g+2$ and $g\geq 4$, then we have the following possibilities:
\begin{enumerate}
\item[(B1)] $r=n-2$, $X$ is a
$\mathbb{P}^{n-1}$-bundle over a smooth curve $B$ and
$\mathcal{E}_F\cong\mathcal{O}_F(1)^{\oplus (n-2)}$ for every fiber
$F\cong\mathbb{P}^{n-1}$ of the bundle projection $\pi :X\to B$.
Moreover, either
\begin{enumerate}
\item[(j)] $H_F=\mathcal{O}_F(3)$, $S=S_{2,-1}$, $g=5$, and $H_S\equiv [3C_0-f]$, or \item[(jj)]
$H_F=\mathcal{O}_F(2)$, $S=S_{q,e}$ with $q\geq 2$, $-q\leq e\leq 0$, $g=2q+1$ and
$H_S\equiv [2C_0+(e+2)f]$;
\end{enumerate}
\item[(B2)] $(X,\mathcal{E},H)$ is as in {\em (i)} and {\em (ii)} of {\em (A)}, and in both cases, $S$ is the blowing-up $\sigma :S\to
S_{q,e}$
of a surface $S_{q,e}$ at two points $p_1,p_2$, lying on distinct fibers, $H_S\equiv [\sigma^*(2C_0+(e+1)f)-\sigma^{-1}(p_1)-\sigma^{-1}(p_2)]$ and $g=2q$.
\end{enumerate}
\end{itemize}
\end{thm}

\begin{proof}

(A) If $g\leq 3$ then $\delta=2g+1\leq 7$ and from Table
\ref{Table1} we conclude that $(\delta,g)=(3,1),(5,2)$,
since $(7,2)$, which corresponds to N. 20,
does not satisfy the current assumption, i.e. the
triplets $(X,\mathcal{E},H)$ fit into all the possibilities of
Theorem \ref{basic2} for $\delta=3$ and $5$. So we can assume $g\geq 4$. By Theorem
\ref{prop 2g+1, 2g+2} we deduce that $g=2q\geq 4$, $S$ is the
blowing-up $\sigma :S\to S_{q,e}$ of a surface $S_{q,e}$ at a
point $p$ and $H_S\equiv [\sigma^*(2C_0+(e+1)f)-\sigma^{-1}(p)]$. Having
in mind \cite[Theorem]{dF} and by arguing as in the proof of
Theorem \ref{basic2}, we can easily deduce cases (i) and (ii) of
the statement.

(B) Since $\delta =2g+2$, from Theorem \ref{prop 2g+1, 2g+2} it
follows that $(S,H_S)$ is one of the following three polarized
surfaces:
\begin{enumerate}
\item[(a)] $S=S_{2,-1}$, $g=5$, $H_S\equiv [3C_0-f]$ and $H_S^2=3$;
\item[(b)] $S=S_{q,e}$ with $q\geq 2$, $g=2q+1$, $H_S\equiv
[2C_0+(e+2)f]$ and $H_S^2=8$;
\item[(c)] $S$ is the blowing-up $\sigma :S \to
S_{q,e}$
of $S_{q,e}$ at two points $p_1,p_2$, lying on distinct fibers, $H_S\equiv [\sigma^*(2C_0+(e+1)f)-\sigma^{-1}(p_1)-\sigma^{-1}(p_2)]$, $g=2q\geq 4$
and $H_S^2=2$.
\end{enumerate}
If $(S,H_S)$ is as in (a) and (b), then by arguing as in cases (1) and (2) of the proof of Proposition \ref{2q}, we obtain (B1) in the statement.
In case (c), recalling \cite[Theorem]{dF} and reasoning as in the proof of Theorem \ref{basic2}, we get (B2) in the statement.
\end{proof}

\begin{thm}\label{2g+d}
Let $(X, \mathcal{E}, H)$ be as in \eqref{0}
and suppose that $S$ is not a ruled surface. Then $\delta\geq
2g+d$.
\end{thm}

\begin{proof}
Simply note that
cases (1) and (2) of Proposition \ref{S not ruled} cannot
ascend to the ample vector bundle setting due to \cite[Theorem]{L2} and
\cite[Theorem]{LM4}. Actually, in the former case $S$ is a minimal surface
of Kodaira dimension zero, while in the latter $S$ is an elliptic surface with multiple fibers.
\end{proof}

\vskip0.5cm

\section{When $H_S$ is ample and spanned or very ample}\label{Sec2}

In this Section, we revisit all the above results in the ample and
spanned (Subsection $4.1$) or very ample (Subsection $4.2$) settings and we
improve some of them.

\medskip

\subsection{$H_S$ is an ample and spanned line bundle}

First of all, note that if $\delta\leq 3$, then the triplets
$(X, \mathcal E, H)$ are as in Theorems \ref{basic1} and \ref{basic2}.
Thus, assume that $\delta >3$.

\begin{thm}\label{basic3}
Let $(X, \mathcal{E}, H)$ be as in \eqref{0}, suppose
that condition \eqref{S} holds, and let $\delta\geq 4$. Then
$$\delta\geq 9,$$
except in the following cases:
\begin{enumerate}
\item $\delta =4$ and $(X, \mathcal E, H)$ is either as in cases $(1)$ and $(2)$ of {\em Theorem \ref{basic2}} $(m_2=12)$, or
$r=n-2$, $X$ is a $\mathbb P^{n-1}$-bundle over a smooth curve $B$ of genus $1$,
$\mathcal E_F = \mathcal O_{\mathbb P^{n-1}}(1)^{\oplus (n-2)}$ and $H_F = \mathcal O_{\mathbb P^{n-1}}(2)$
for every fiber  $F = \mathbb P^{n-1}$  of the bundle projection $X\to B$ and $(S,H_S)\cong (S_{1,-1},[2C_0])$ $(m_2=8)$;
\item $\delta = 6$ and $(X, \mathcal E, H)\cong (\mathbb{P}^2\times\mathbb{P}^2,\mathcal{O}_{\mathbb{P}^2\times\mathbb{P}^2} (1,1)^{\oplus 2},\mathcal{O}_{\mathbb{P}^2\times\mathbb{P}^2} (1,1))$; \hfill $(m_2=12)$
\item $\delta = 7$ and we have either $(X, \mathcal E, H)\cong (\mathbb{Q}^4,
\mathcal{S}\otimes\mathcal{O}_{\mathbb{Q}^4}(2),\mathcal{O}_{\mathbb{Q}^4}(1))$,
where $\mathcal{S}$ is a spinor bundle on $\mathbb{Q}^4$,
or $X$ is a linear section of the Grassmannian variety $\mathbb{G}(1,4)\subset\mathbb{P}^9$
and $(\mathcal{E},H)\cong (L^{\oplus r},L)$, where $L$ is the
ample generator of {\em Pic}$(X)$;
 \hfill $(m_2=12)$
\item $\delta = 8$ and $(X, \mathcal E, H)$ is one of the following triplets:
\begin{enumerate}
\item[$(a)$] $(\mathbb P^n, \mathcal O_{\mathbb P^n}(2)^{\oplus 2} \oplus \mathcal O_{\mathbb P^n}(1)^{\oplus (r-2)}, \mathcal O_{\mathbb P^n}(1))$; \hfill $(m_2=12)$
\item[$(b)$] $(\mathbb Q^n, \mathcal O_{\mathbb Q^n}(2) \oplus \mathcal O_{\mathbb Q^n}(1)^{\oplus (r-1)}, \mathcal O_{\mathbb Q^n}(1))$; \hfill $(m_2=12)$
\item[$(c)$] $X$ is a complete intersection of two quadric hypersurfaces of $\mathbb{P}^{n+2}$
and $(\mathcal{E},H)\cong (L^{\oplus r},L)$, where $L$ is the
ample generator of {\em Pic}$(X)$; \hfill $(m_2=12)$ \item[$(d)$]
$r=n-2$ and there is a vector bundle $\mathcal{V}$ on a smooth
curve $C$ of genus $q\leq 2$ such that
$X\cong\mathbb{P}_C(\mathcal{V})$, $H_F=\mathcal{O}_{F}(2)$ and
$\mathcal{E}_F\cong\mathcal{O}_F(1)^{\oplus (n-2)}$ for any fiber
$F\cong\mathbb{P}^{n-1}$ of $X\to C$; moreover, $(S,H_S)$ is, up to numerical equivalence, one
of the following pairs:
\begin{enumerate}
\item[$(d_1)$] $(S_{0,e},[2C_0+(e+3)f])$, with $0\leq e\leq 2$; \hfill $(m_2=20)$
\item[$(d_2)$] $(S_{1,e},[2C_0+(e+2)f])$, with $-1\leq e\leq 0$; \hfill $(m_2=16)$
\item[$(d_3)$] $(S_{2,e},[2C_0+(e+1)f])$, with $-2\leq e\leq -1$.
\hfill $(m_2=12)$
\end{enumerate}

\end{enumerate}
\end{enumerate}

\end{thm}

\begin{proof}
If $\delta =4$, then from Theorem \ref{basic2} it follows case (1)
of the statement. Actually, by \cite[Theorem]{LP2} the remaining possibilities
in Theorem \ref{basic2}\ (3) cannot occur, $H_S$ being spanned.
If $\delta =5$, then by Lemma \ref{doublecovers}
and Table \ref{Table1} we deduce that the only possible cases for
$(S,H_S)$ are N. 12 and 15
of Table \ref{Table1}. The former
case does not lift to the vector bundle setting by \cite[Theorem 4
and Remark in $\S 2$]{LM3} and the latest one cannot occur since
$g=2$ and $H_S$ is required to be ample and spanned (see
\cite[Theorem (3.1)]{LP}). This shows that $\delta =5$ cannot
occur.
If $\delta =6,7$ we have $g=1$ or $2$, by Table \ref{Table1}. If $g=1$, from Lemma
\ref{doublecovers} and \cite[Theorem 4 and Remark in $\S 2$]{LM3}
we obtain immediately cases (2) and (3) of the statement. On the other
hand it cannot be $g=2$ because $H_S$ is ample and spanned: actually,
in cases N. 17, 18 and 20
of Table \ref{Table1}, $S$ is not a minimal surface
and this is not compatible with \cite[Theorem (3.1)]{LP} again.
Suppose
now that $\delta =8$. First of all, assume that $\kappa(S)\geq 0$.
Then $(S,H_S)$ is as in cases N. 22, 23 and 24 of Table
\ref{Table1}. Since $H_S$ is ample and spanned and $g=3$ in all
cases, by \cite[Table I, p. 268]{LL} we see that N. 22 and 23
cannot occur and that in N. 24 the
surface $S$ is a minimal elliptic fibration with multiple fibers,
but this situation does not lift to the vector bundle setting by
\cite[Theorem]{LM4}. Finally, suppose that
$\kappa(S)=-\infty$, i.e. $S$ is a ruled surface. From Table
\ref{Table1} we deduce that either
\begin{enumerate}
\item[(i)] $H_S^2=4, g=1, e(S)=8$,
\end{enumerate}
or $(S,H_S)$ is,
up to numerical equivalence, one of the following
pairs:
\begin{enumerate}
\item[(ii)] $(S_{0,e},[2C_0+(3+e)f])$ with $e=0,1,2$, $H_S^2=12, g=2, e(S)=4$;
\item[(iii)] $(S_{1,e},[2C_0+(2+e)f])$ with $e=-1,0,1$, $H_S^2=8, g=3, e(S)=0$;
\item[(iv)] $(S_{1,0},[3C_0+f])$ with $H_S^2=6, g=3, e(S)=0$;
\item[(v)] $(S_{1,-1},[5C_0-2f])$ with $H_S^2= 5, g=3, e(S)=0$;
\item[(vi)] $(S_{2,e}, [2C_0+(e+1)f])$ with $-2\leq e\leq 0$, $H_S^2=4, g=4, e(S)=-4$.
\end{enumerate}
Note that (iv) cannot occur since in this case $H_S\cdot C_0=1$ with $g(C_0)=1$
implies that $H_S$ is not spanned. Moreover, from \cite[Table II, p.\ 268]{LL} it follows that
also case (v) is not possible since $0=e(S)=12(1-q)-K_S^2$. In case (i), by
\cite[Theorem 4 and Remark in $\S 2$]{LM3} we get cases $(a), (b)$
and $(c)$ of the statement. Finally, having in mind that $H_S$ is
ample and spanned, cases (ii), (iii) and (vi) lead to cases
$(d_1),(d_2)$ and $(d_3)$ of the statement by \cite{LM2}.
\end{proof}

\begin{cor}\label{basic4}
Let $(X, \mathcal{E}, H)$ be as in \eqref{0} and suppose that
condition \eqref{S} holds. Then
$$m_2\leq 11$$ if and only if either $(X,\mathcal
E,H)\in\mathcal{A}$\ $(m_{2}=0,3,2,2,d\leq 11)$, or $r=n-2$, $X$
is a $\mathbb P^{n-1}$-bundle over a smooth curve $B$ of genus
$1$, $\mathcal E_F = \mathcal O_{\mathbb P^{n-1}}(1)^{\oplus
(n-2)}$ and $H_F = \mathcal O_{\mathbb P^{n-1}}(2)$ for every
fiber  $F = \mathbb P^{n-1}$  of the bundle projection $X\to B$
and $(S,H_S)\cong (S_{1,-1},[2C_0])$\ $(m_{2}=8)$.
\end{cor}

\begin{proof}
Since $m_2=\delta + H_S^2$ and $H_S^2\geq 3$ unless a few exceptions for the pairs $(S,H_S)$ described in Lemma \ref{doublecovers},
the result follows from Theorems \ref{basic1}, \ref{basic2} and \ref{basic3}.
\end{proof}

\begin{cor}
Let $(X, \mathcal{E}, H)$ be as in \eqref{0} and suppose that
\eqref{S} holds. Then $$\delta\geq 2g+3$$ unless either $g\geq 4$
and $(X,\mathcal E,H)$ is as in {\em Proposition \ref{2q}} and
{\em Theorem \ref{2g+1, 2g+2}}, or $g\leq 3$ and one of the
following cases occurs:
\begin{enumerate}
\item $(X,\mathcal E,H)\in\mathcal{A}\cup\mathcal{B}$\hfill
$(m_{2}=0,3,2,2,d,12,12,12)$;
\item $r=n-2$, $X$ is a $\mathbb
P^{n-1}$-bundle over a smooth curve $B$ of genus $1$, $\mathcal
E_F = \mathcal O_{\mathbb P^{n-1}}(1)^{\oplus (n-2)}$ and $H_F =
\mathcal O_{\mathbb P^{n-1}}(t)$, with $t=2$ or $3$, for every
fiber $F = \mathbb P^{n-1}$  of the bundle projection $X\to B$;
moreover, $(S,H_S)$ is, up to numerical equivalence, either
$(S_{1,-1},[3C_0-f])\ (m_{2}=7)$ or $(S_{1,e},[2C_0+(e+1)f])$ with
$e\in\{-1,0\}$\hfill
 $(m_{2}=8)$;
\item $r=n-2$ and there is a vector bundle $\mathcal{V}$ on a
smooth curve $C$ of genus $q\leq 2$ such that
$X\cong\mathbb{P}_C(\mathcal{V})$, $H_F=\mathcal{O}_{F}(2)$,
$\mathcal{E}_F\cong\mathcal{O}_F(1)^{\oplus (n-2)}$ for any fiber
$F\cong\mathbb{P}^{n-1}$ of $X\to C$ and $(S,H_S)$ is, up to numerical equivalence,
$(S_{1,e},[2C_0+(e+2)f])$ with $e\in\{-1,0\}$. \hfill $(m_2=16)$
\end{enumerate}
\end{cor}

\begin{proof}
Let $\delta\leq 2g+2$. If $g\leq 3$, then $\delta\leq 8$ and the assertion follows from Theorems \ref{basic1}, \ref{basic2} and \ref{basic3}.
If $g \geq 4$, then Proposition \ref{2q} and Theorem \ref{2g+1, 2g+2} apply.
\end{proof}

\medskip

When $S$ is not a ruled surface, we have the following two results.

\begin{prop}\label{basic5}
Let $(X, \mathcal{E}, H)$ be as in \eqref{0} and suppose that \eqref{S} holds.
If $S$ is not a ruled surface, then $\delta\geq 2g+5$.
\end{prop}
\begin{proof}
Recall that $\delta=m_2-d$ and assume, by contradiction, that
$\delta\leq 2g+4$. If $d \leq 2$, then from Lemma
\ref{doublecovers} it follows that $d=2$, $g=b-1$ and
$e(S)=2(2b^2-3b+3)$ for some integer $b\geq 3$, $S$ being not a
ruled surface. Then
$$\delta = e(S)+4(g-1) = 2g + [2b-6+2(2b^2-3b+3)] \geq 2g+24,$$
which is impossible. Moreover, if $g \leq 2$ then $g=2$ and by \cite[Theorem (3.1)]{LP}
the only possibility for $S$ is to be a K3
surface, in which case however, $\delta = e(S)+4(g-1)=
2g+e(S)+2g-4 = 2g + 24$. So we can assume that $d \geq 3$ and
$g\geq 3$. Note that by Theorem \ref{2g+d} cases $\delta = 2g+1,
2g+2$ do not occur. Therefore, it is enough to show that also
cases $\delta=2g+3,2g+4$ cannot occur.

\medskip

First suppose that $\delta=2g+3$. Then from $(\#)$ we deduce that
$e(S) + 2g = 7$ and then $(e(S),g) = (1,3)$ since $e(S) \geq 0$, $S$
being not a ruled surface, and $g \geq 3$. Thus $4 = 2g-2 =
d+H_SK_S \geq 3+ H_SK_S$, hence $H_SK_S \leq 1$. It cannot be
$H_SK_S=0$, otherwise, $K_S$ would be numerically trivial, due to
the ampleness of $H$, but in this case $S$, could not satisfy
$e(S)=1$, in view of the classification. Therefore $H_SK_S=1$ and
by Lemma \ref{lemma} we see that $\kappa(S)=0$. Moreover,
$(S,H_S)$ has $(S_0,L_0)$ as simple reduction, where $S_0$ is a
minimal surface of Kodaira dimension zero with $e(S_0)=0$. So
$S_0$ is either abelian or bielliptic, and therefore
$\chi(\mathcal O_{S_0})=0$. On the other hand, with the same
notation as in Lemma \ref{lemma}, $H_S = \sigma^*L_0 - E$,
$\sigma:S \to S_0$ being the reduction morphism contracting the
exceptional curve $E$ at $p \in S_0$. We have $h^0(H_S) = h^0(L_0)
- \varepsilon$ where $\varepsilon = 0$ or $1$ according to whether
$p$ is a base point of $|L_0|$ or not. Then, due to the
spannedness of $H_S$, by the Riemann--Roch theorem and the Kodaira
vanishing theorem we get
$$3 \leq h^0(H_S) = h^0(L_0) - \varepsilon = \chi(\mathcal O_{S_0})+ \frac{1}{2}L_0^2 -\varepsilon
= \chi(\mathcal O_{S_0}) + 2 - \varepsilon.$$ This gives
$0= \chi(\mathcal O_{S_0})=1+ \varepsilon \geq 1$, a contradiction.

\medskip

Finally, suppose that $\delta = 2g+4$. By arguing as in case
$\delta = 2g+3$, we get only two possibilities for $(e(S),g)$,
namely,
\begin{enumerate}
\item[a)] $(2,3)$, or \item[b)] $(0,4)$.
\end{enumerate}

\noindent In case a), by using ($\#$) again,
we see that $e(S)+ 2(g-3) = 2$, hence $(H_S^2, H_S K_S)$ is either
$(3,1)$ or $(4,0)$, by genus formula. Both possibilities rule out.
Actually, in the latter case $K_S$ would be numerically trivial,
but this cannot occur for $e(S)=2$. In the former case, $S$ could
be minimal with $\kappa(S)=1$, but then $K_S^2=0$, which
contradicts condition $e(S)=2$ in view of Noether's formula. So
$S$ is not minimal. Thus Proposition \ref{lemma} implies that
$\kappa(S)=0$ and then $e(S)=2$ says that $S$ is an abelian or a
bielliptic surface blown-up at two points. Then $K_S \equiv E$
where $E$ consists of two irreducible curves, hence $1 = H_S K_S =
E K_S \geq 2$, a contradiction.

Now consider case b). By using the facts that $d \geq 3$ and $H_S
K_S \geq 0$, we get for $(d, H_S K_S)$ the following list of
possible values: $(6,0), (5,1), (4,2), (3,3)$. If $H_S K_S = 0$
(first case), recalling that $e(S)=0$ we conclude that $S$ is
either an abelian or a bielliptic surface. Both cases do not
ascend the ample vector bundle setting due to \cite[Theorem]{L2}.
If $H_S K_S = 1$ (second case), then
Lemma \ref{lemma} ii)
implies that $\kappa(S)=0$ and then $e(S)=0$ allows us to conclude
that $S$ is an abelian or a bielliptic surface; but then we get $0
= H_S K_S = 1$, a contradiction. Next let us deal with the third
and the fourth cases at the same time. Since $e(S)=0$ we have that
either i) $S$ is an abelian or a bielliptic surface, or ii) $S$ is
a minimal elliptic fibration. In subcase i) $K_S$ is numerically
trivial, hence $0 = H_S K_S = 2$ or $3$, a contradiction. In
subcase ii) we have $K_S^2=0$. This combined with the fact that
$e(S)=0$ implies $\chi(\mathcal O_S)=0$, by Noether's formula. Now
use \cite[table in Proposition 4.4 and Proposition 1.4]{LP2}. For
$d=3$, since $S$ is a minimal elliptic surface, \cite[Proposition
1.4, case (1.4.2), (i)]{LP2} shows that $\chi(\mathcal O_S) = 3$,
a contradiction. On the other hand, for $d=4$, since
$\chi(\mathcal O_S) = 0$, \cite[table in Proposition 4.4]{LP2}
shows that necessarily the elliptic fibration of $S$ has some
multiple fibers. Therefore this case does
not ascend to the ample vector bundle
setting in view of \cite[Theorem]{LM4}.
\end{proof}

\begin{thm}\label{main}
Let $(X, \mathcal{E}, H)$ be as in \eqref{0} and suppose that
\eqref{S} holds.
If $S$ is not a ruled surface and $\delta = 2g+5$, then
$X = \mathbb P_{S_0}(\mathcal V)$, where $\mathcal V$ is a vector bundle of rank $(n-1)$ over a smooth
minimal surface $S_0$, which is either abelian or bielliptic; moreover, $r=n-2$ and $\mathcal E = \pi^* \mathcal G \otimes \xi$,
where $\xi$ is the tautological line bundle on $X$, $\mathcal G$ is a vector bundle of rank $n-2$ on $S_0$ and
$\pi:X \to S_0$ is the bundle projection; furthermore, $\pi|_{Z}: Z \to S_0$ is a birational morphism expressing
$Z$ as $S_0$ blown up a single point, say $p$. Finally, $H= 2\xi + \pi^*\big(A-2(\det \mathcal V + \det \mathcal G) \big)$, where
$A$ is an ample and spanned line bundle on $S_0$ with $A^2=8$ and $p$ belongs to its second jumping set $\mathcal J_2(S_0, A)$.
\end{thm}

For the definition of the jumping sets of an ample and spanned line bundle we refer to \cite{LPS}.

\begin{proof}
In view of ($\#$), the relation $\delta=m_2-d=2g+5$ converts into
\begin{equation}\label{2g+5}
e(S)+2(g-3)=3.
\end{equation}
We have $e(S) \geq 0$ by the Castelnuovo--de Franchis Theorem \cite[Theorem X.4]{B1}, hence $g \leq 4$.
Clearly it cannot be $g \leq 1$, since $S$ is not a ruled surface. Moreover, for $g=2$, the only pair
$(S,H_S)$ with $S$ a not ruled surface is the K3 double plane, according to the classification in
\cite[Theorem 3.1]{LP}, but in this case $e(S)=24$,
which contradicts (\ref{2g+5}).

Suppose that $g=3$; then $e(S)=3$ by (\ref{2g+5}). Clearly $H_S^2 \geq 2$ and taking into account Lemma \ref{doublecovers}
we see that $e(S)=3$ is not compatible with $H_S^2=2$. Thus the genus formula, combined with the fact that $S$ is not a ruled surface, implies
$(H_S^2, H_SK_S) = (4,0)$, or $(3,1)$. In the former case $K_S$ is numerically trivial, hence $S$ is a minimal surface with Kodaira dimension $\kappa(S)=0$,
but this contradicts $e(S)=3$. In the latter case the Hodge index theorem shows that $K_S^2 \leq 0$. Suppose that $S$ is minimal.
Thus $K_S^2=0$, since $\kappa(S) \geq 0$, but then Noether's formula contradicts $e(S)=3$ again. Therefore $S$ is not minimal.
Let $\eta:S \to S_0$ be a birational morphism to the minimal model. We know that $K_S = \eta^* K_{S_0} + E$, where $E$ is an
effective divisor contracted by $\eta$ to a finite set. Consider the equality $1 = H_S K_S = H_S \eta^* K_{S_0} + H_S E$:
the second summand on the right hand is greater than or equal to the number of blowing-ups $\eta$ factors through;
on the other hand, the first one is non-negative and it is zero if and only if
$K_{S_0}$ is numerically trivial. It follows that $S$ is $S_0$ blown up at a single point, $E$ being the corresponding exceptional curve, and $\kappa(S)=0$.
But then $e(S)=e(S_0)+1\not= 3$, a contradiction. Thus $g=3$ cannot occur as well.

It remains to consider the case $(e(S),g) = (1,4)$. Clearly $H_S^2 \geq 2$ and by
Lemma \ref{doublecovers} we see that
condition $e(S)=1$ is not compatible with $H_S^2=2$, as before. Thus the genus formula, combined with the fact that $S$ is not a ruled surface, implies
$3 \leq H_S^2 \leq 6$. A close inspection of \cite{LP2} shows that it cannot be $e(S)=1$ if $H_S^2 = 3,5$ or $6$.
Actually, as observed before, $S$ cannot be of general type, hence it has Kodaira dimension $\kappa(S)=0$ or $1$.
According to \cite[Proposition 1.4, Lemma 2.1 combined with Proposition 2.3, and Proposition 3.1]{LP2}, we see that
condition $e(S)=1$ would be contradicted. So, $H_S^2=4$. Now, from \cite[Proposition 1.6]{LP2} we easily see that
$e(S)=1$ can occur only when $S$ is a $4$-tuple cover of $\mathbb P^2$ via $|H_S|$, i.e. $h^0(H_S)=3$.
Clearly condition $e(S)=1$ prevents $S$ from being a minimal surface.
Thus, if $\kappa(S)=1$ \cite[Proposition 4.3]{LP2}
would imply that $(S,H_S)$ is obtained by blowing up a single point on a minimal elliptic surface with $q=0$.
Thus $\chi(\mathcal O_S) \geq 1$.
By Noether's formula we have $K_S^2 + 1 = K_S^2 + e(S) \geq 12$, hence $K_S^2 \geq 11$. But this is not compatible
with $\kappa(S)=1$.
This check settles all possibilities,
except when $\kappa(S)=0$ and $H_S^2 = 4$, in which case $H_S K_S=2$ by genus formula.
Let $\eta:S \to S_0$ be a birational morphism from $S$ to its minimal model $S_0$. Since $e(S)=1$,
$\eta$ is simply the blowing-up at a point $p \in S_0$; in particular, we get $e(S_0)=0$, hence the surface $S_0$
is either abelian or bielliptic. From $2=K_S H_S = (\eta^*K_{S_0} + E) H_S$, where $E=\eta^{-1}(p)$ is the
exceptional curve, we see that $H_S E=2$, $K_{S_0}$ being numerically trivial;
hence $H_S = \eta^*A - 2E$, where $A$ is an ample line bundle on $S_0$, and $4 = H_S^2 = A^2 - 4$, i.e.
$A^2=8$. Thus
$$h^0(A) = \chi(\mathcal O_{S_0}) + \frac{1}{2}(A^2 - AK_{S_0}) = 4,$$
by the Riemann--Roch and the Kodaira vanishing theorems.
Moreover, since $A^2=8$ it follows from Reider's Theorem \cite[Theorem 1]{R} that $A$ is also a spanned line bundle.
According to the above, $|H_S|$ is in bijection with the linear system $|A-2p|$ of divisors in $|A|$
having a double point at $p$. Recalling that $h^0(H_S)=3$, this shows that
$$3 = h^0(H_S) = h^0(A) - \sharp = 4 - \sharp ,$$
where $\sharp$ stands for the number of linearly independent linear conditions to be imposed
on the elements of $|A|$ in order to have a double point at $p$.
Therefore $\sharp = 1$. This says that
$\text{codim}_{|A|}(|A-2p|) = 1$. Thus, the spannedness of $A$ implies that
$|A-2p| = |A-p|$, i.e. the point $p$ is in the second jumping set $\mathcal J_2(S_0,A)$.
Now come back to the ample vector bundle setting.
By using \cite[Theorem]{L2}, we conclude that $X$ is as in the statement with $\mathcal F =  \pi^* \mathcal G \otimes \xi$,
where $\xi$ is the tautological line bundle on $X$, $\mathcal G$ is a vector bundle of rank $n-2$ on $S_0$ and
$\pi:X \to S_0$ is the bundle projection; moreover, $\pi|_{S}: S \to S_0$ is just the birational morphism $\eta$ expressing
$S$ as $S_0$ blown up the single point $p$.
Now consider $H$. If $r < n-2$, then $H_F$ is a summand of $\mathcal F_F$, hence $H_F = \mathcal O_{\mathbb P^{n-2}}(1)$.
Since $E$ is contained in a fiber $F$ of $\pi$
and is a line with respect to $\xi_F$, we get the contradiction $1 = \deg (H_F)_E = H_S E = 2$.
Therefore $r=n-2$, hence $\mathcal F = \mathcal E$ and $S=Z$, so that $Z$ is as in the statement. Since $H$ is ample we have
$H_F = t \xi _F = \mathcal O_{\mathbb P^{n-2}}(t)$ for some positive integer $t$,
and then we see from the equality $t = \deg (H_F)_E = H_S E = 2$ that $H_F = 2 \xi_F$. So, $H = 2\xi + \pi^* \mathcal M$
for some line bundle $\mathcal M$ on $S_0$, which we have to determine.
Recall that $H_S = \pi|_S^* A - 2E$, where
$A$ is the ample and spanned line bundle on $S_0$ with $A^2=8$ we met before.
Now, by adjunction $K_S = (K_X+\det \mathcal F)_S$ and then by the canonical bundle formula we get
$$K_S=\big(-(n-1)\xi + \pi^*(K_{S_0}+\det \mathcal V) + (n-2)\xi + \pi^* \det \mathcal G \big)_S=-\xi_S + \pi|_S^*(K_{S_0}+ \det \mathcal V + \det \mathcal G).$$
On the other hand, $K_S = \pi|_S^* K_{S_0}+E$, which provides the expression of $E$; hence
$$H_S = \pi|_S^*A - 2E
= \pi|_S^*A - 2\big(-\xi_S + \pi|_S^* (\det \mathcal V + \det \mathcal G) \big)
= 2\xi_S + \pi|_S^*\big( A - 2(\det \mathcal V + \det \mathcal G) \big).$$
Finally, from the injectivity of the restriction homomorphism
$\text{Pic}(X) \to \text{Pic}(S)$ (Lefschetz--Sommese Theorem), we get the
expression of $H$ as in the statement.
\end{proof}

\smallskip

\begin{rem} \label{remark 4.6}
{\em i) We want to stress that Theorem \ref{main} is effective. To see this it is enough to modify the example
produced in Remark \ref{Rem-example}, as follows.
Let $X$ be the Jacobian bundle $\pi: \mathbb P(\mathcal E_{n-1}(C,o))
\to S_0=J(C)$ on the jacobian of a smooth curve $C$ of genus $2$ again, and call $\xi$ the tautological line bundle.
Letting $\mathcal E= \xi^{\oplus (n-2)}$ and taking $H=2\xi$, we see that $H$ is ample
and spanned \cite[Example 5.1]{FL},
and the triplet $(X,\mathcal E, H)$ is as in Theorem \ref{main}:
here $S=C^{(2)}$ again, but
$A$ is the line bundle corresponding to the double of the curve $C$ itself embedded in its jacobian.
Unfortunately, we have no examples with $S_0$ a bielliptic surface.}

\smallskip

\noindent {\em ii) According to the discussion in the first part of the
proof we have to stress a gap affecting the proof of
\cite[Proposition 4.5]{LP2}. Actually, the equality in the first
case of (4.5.1) of \cite[p.\ 101]{LP2} holds provided that the
point $p$ does not belong to the first jumping set of $L'$ (see
\cite[$\S 1$]{LPS}): to wit, set $\mathcal J_i:= \mathcal
J_i(X',L')$, for $i=0,1,2$, where $\mathcal J_0 = X \setminus
\mathcal J_1$; using the same notation as there, if $L'$ is
spanned then the mentioned equality has to be amended as follows:
$h^0(L')=h^0(L)+3-i$, where $p \in \mathcal J_i$. As a
consequence, pairs $(X',L')$ with $X'$ an abelian or a bielliptic
surface when $p \in \mathcal J_2$ and with $X'$ an Enriques
surface when $p \in \mathcal J_1 \setminus \mathcal J_2$ are not
ruled out.}
\end{rem}

\medskip

\subsection{Revisiting the classical setting}

As a consequence of Remark \ref{veryample}, revisiting Theorem \ref{basic3} and Corollary \ref{basic4}, we obtain the following two results.

\begin{cor}
Let $(X, \mathcal{E}, H)$ be as in \eqref{0}, suppose
that condition \eqref{VA} holds, and let $\delta\geq 4$. Then
$$\delta\geq 9,$$
except in the following cases:
\begin{enumerate}
\item $\delta =4$ and $(X, \mathcal E, H)$ is as in cases $(1)$ and $(2)$ of {\em Theorem \ref{basic2}} $(m_2=12)$;
\item $\delta = 6, 7$ and $(X, \mathcal E, H)$ is as in cases $(2)$ and $(3)$ of {\em Theorem \ref{basic3}}, respectively $(m_2=12)$;
\item $\delta = 8$ and $(X, \mathcal E, H)$ is as in cases $(4)(a),(b),(c)$ $(m_2=12)$ and cases $(4)(d_1)$ $(m_2=20)$ and $(4)(d_2)$ with $e=-1$ of {\em Theorem \ref{basic3}} $(m_2=16)$.
\end{enumerate}
\end{cor}

\begin{cor}
Let $(X, \mathcal{E}, H)$ be as in \eqref{0} and suppose that
condition
\eqref{VA} holds. Then
$$m_2\leq 11$$ if and only if $(X,\mathcal
E,H)\in\mathcal{A}$\ $(m_{2}=0,3,2,2,d\leq 11)$.
\end{cor}

Recently, Fukuma \cite{Fu} improved a result of the first author, showing the following
\begin{prop} \label{Fukuma}
Let $S$ be a smooth surface endowed with a very ample line bundle $L$, and let $d, g, m$ be the degree, the sectional genus and the class of $(S,L)$. Suppose that $m > d$ and $g \geq 2$. Then $m \geq d + 2g+2$ and equality holds if and only if
$(S,L)=(S_{1,-1},[2C_0+f])$ (in which case $d=8$, $g=3$).
\end{prop}

Note that the above pair $(S,L)$ corresponds to N. 26
with $e=-1$ in Table \ref{Table1}. In particular, it fits into case (B)$(\gamma)$ with $(m-d,g)=(8,3)$ of Theorem \ref{prop 2g+1, 2g+2}.
Coming back to triplets as in (\ref{0}), observe that for $g:=g(S,H_S)\leq 1$, condition $\delta\leq 2g+2$ simply means
$\delta\leq 4$. Then taking into account Theorems \ref{basic1}, \ref{basic2} and \ref{basic3}, the very ampleness of $H_S$ implies that $(X, \mathcal{E},H)\in\mathcal{A}\cup \mathcal{B}$. Thus we can assume that $g\geq 2$ and so Proposition \ref{Fukuma} can be easily lifted to the ample vector bundle setting, as follows.

\begin{prop} \label{AVB}
Let $(X, \mathcal{E},H)$ be as in \eqref{0} and suppose that
\eqref{VA} holds.
Assume $g:=g(S,H_S)\geq 2$ and $\delta > 0$. Then $$\delta \geq 2g+2$$ and equality holds if and only if
$r=n-2$,
$X$ is a $\mathbb P^{n-1}$-bundle over an elliptic curve $B$, $\mathcal E_F =\mathcal O_{\mathbb P^{n-1}}(1)^{\oplus (n-2)}$,
$H_F = \mathcal O_{\mathbb P^{n-1}}(2)$ for every
fiber $F \cong \mathbb P^{n-1}$ of the projection $X \to B$, and $(S,H_S)$ is the pair $(S,L)$ described in {\em Proposition \ref{Fukuma}}.
\end{prop}

\begin{proof}
Note that $(S,H_S)$ satisfies all the
assumptions of Proposition \ref{Fukuma} with $L=H_S$.
This implies the claimed inequality. Now suppose that equality holds; then $(S,H_S)$
is the pair $(S,L)$ described in Proposition \ref{Fukuma}. Set $\mathcal F = \mathcal E \oplus H^{\oplus (n-r-2)}$.
Since $S$ is a $\mathbb P^1$-bundle over an
elliptic curve, say $B$, we can conclude by \cite{LM2} that $X$ is a $\mathbb P^{n-1}$-bundle over
$B$, the projection $p: X \to B$ inducing the ruling of $S$, and
$\mathcal F_F =\mathcal O_{\mathbb P^{n-1}}(1)^{\oplus (n-2)}$ for every fiber
$F$. It cannot be $r < n-2$, since $(S,H_S)$ is not a scroll. Therefore $r=n-2$, and then
$H_F = \mathcal O_{\mathbb P^{n-1}}(2)$, since $H_S f = (2C_0+f) f =2$.
The converse is obvious and this concludes the proof.
\end{proof}

\begin{rem}
{\em Proposition \ref{Fukuma} is effective, since $(S,L)$
is a very well known elliptic conic bundle in $\mathbb P^5$.
We want to stress that Proposition \ref{AVB} is
effective as well. Arguing as in \cite[Section 3]{LS} we can produce an example.
Let $\mathcal{V}_n$ be an indecomposable vector
bundle of rank $n$ and degree $1$ over the elliptic curve $B$, and set $X :=
\mathbb{P}(\mathcal{V}_n)$. We note that any two such bundles
$\mathcal{V}_n$, $\mathcal{V}^{\prime}_n$ are related by
$\mathcal{V}_n = \mathcal{V}^{\prime}_n \otimes \tau$, where
$\tau$ is a line bundle of degree $0$ on $B$. Thus $X$ is the same
for all choices of $\mathcal{V}_n$. We also note that any such
vector bundle $\mathcal{V}_n$ can be constructed inductively from
a non--split exact sequence $0 \rightarrow \mathcal{O}_B
\rightarrow \mathcal{V}_n  \rightarrow \mathcal{V}_{n-1}
\rightarrow 0$, starting from a line bundle $\mathcal{V}_1$ of
degree $1$. We have $h^0(\mathcal{V}_n) = 1$ for all $n\geq
1$, hence the tautological line bundle $\xi$ on $X$ has a single
section (up to a nonzero constant factor). Since the section of
$\mathcal{V}_n$ vanishes nowhere on $B$, it follows that the
corresponding section of $\xi$ vanishes exactly on
$\mathbb{P}(\mathcal{V}_{n-1})$. Note also that $\mathcal{V}_n$ is
ample for any $n \geq 1$. Hence $\xi$ is ample. Now let $\xi_1$,
..., $\xi_{n-2}$ be $\xi$ twisted by the pullbacks on $X$ of $n-2$
distinct degree $0$ line bundles on $B$ and let $\mathcal{E} =
\bigoplus_{i=1}^{n-2} \xi_i$. Then $\mathcal{E}$ is an ample
vector bundle on $X$. Consider its section $s=(s_1, \dots ,
s_{n-2})$ where $\langle s_i \rangle = H^0(\xi_i)$ and let $Z$ be
its zero locus. Then $Z \cong \mathbb P({\mathcal V}_2)$ \cite[Claim B]{LS},
i.e., $Z$ is the $\mathbb P^1$-bundle of invariant $-1$ over $B$,
and $\xi_Z = [C_0]$, $C_0$ being the tautological section. Now, letting
$H := 2\xi + F$ we have that $H$ is an ample
line bundle, since $\xi$ is ample and $F$ is nef;
Moreover $H_Z=[2C_0+f]$ is very ample, due to Reider's theorem \cite[Theorem 1]{R}}.
\end{rem}

Assuming that $S$ is not a ruled surface, assumption \eqref{VA} allows us to improve Proposition \ref{basic5}, probably roughly, as follows.

\begin{cor}
Let $(X,\mathcal{E},H)$ be as in
\eqref{0} and suppose that \eqref{VA} holds. If $S$ is not
a ruled surface, then $\delta\geq 2g+11$.
\end{cor}
\begin{proof}
Assume that $\delta\leq 2g+10$. Then by Theorem \ref{2g+d}
we see that $d\leq 10$. Furthermore, we have also
$$2g+10\geq \delta = m_2-d=e(S)+4(g-1)\geq 4(g-1),$$ i.e. $g\leq 7$.
On the other hand, since $S$ is not ruled, \eqref{VA} implies $g \geq 3$.
This allows us to use
\cite[Table in (4.0)]{Li}.
We can write
$$\delta=2g+(e(S)+2g-4)=2g+D,$$
where $D:= e(S)+2g-4 = 12\chi(\mathcal{O}_S)-K_S^2+2g-4$ by Noether's formula.
Table in \cite[(4.0)]{Li} shows that $S$ is birational to a K3 surface for $3 \leq g\leq 5$, but
in this case $D\geq 24-K_S^2+2\geq 26$, a contradiction.
On the other hand, if $g=6$ and $S$ is of general type, then $D=63$, while in the remaining cases $D\geq 12-K_S^2+8\geq 20$,
except when $S$ is either an abelian or a bielliptic surface (Cases 8) and 9) in the Table), but in these two cases
$S$ is minimal and this possibility is ruled out by \cite[Theorem]{L2}.
Finally, for $g=7$ condition $d\leq 10$ prevents $S$ from being birational to an abelian or a bielliptic surface (Cases
23) and 25) in the Table)
and in the remaining cases we have $D\geq 12-K_S^2+10\geq 21$, a contradiction.
\end{proof}

\bigskip

{\bf Acknowledgements}: The first author is a member of G.N.S.A.G.A. of the Italian INdAM. He would like to thank
 the PRIN 2010-11 Geometry of Algebraic Varieties and the University of Milano for partial support.
 During the preparation of this paper, the second author was partially supported by the National Project Anillo ACT
 1415 PIA CONICYT and the Proyecto VRID N.214.013.039-1.OIN of the University of Concepci\'on.
 The authors are grateful to the referee for accurate reading of the manuscript and for several useful remarks.

\end{document}